\documentclass[11pt]{amsart}
\usepackage{bbm}
\usepackage{amsmath, amssymb, amsthm,color}
\usepackage{amssymb,longtable}
\usepackage{amsfonts}
\usepackage{amsmath} 
\usepackage{hyperref} 
\usepackage{latexsym}
\usepackage{array}
\usepackage{amssymb}
\usepackage{psfrag}
\theoremstyle{plain}

\numberwithin{equation}{section}
\newtheorem{theorem}{Theorem}[section]
\newtheorem{proposition}[theorem]{Proposition}
\newtheorem{lemma}[theorem]{Lemma}
\newtheorem{corollary}[theorem]{Corollary}

\newtheorem{property}[theorem]{Property}

\theoremstyle{remark}
\newtheorem{remark}[theorem]{Remark}

\newcommand{\mb}{\medskip\noindent}

\newcommand{\R}{\mathbb R}

\def \vsp {\vspace{3pt}}

\begin{document}
\
\vspace{-1cm}

\title[Bilinear dispersive estimates]{Bilinear dispersive estimates via space time resonances. Part II: dimensions two and three}

\date{\today}

\author[F. Bernicot]{Fr\'ed\'eric Bernicot}
\address{F. Bernicot, CNRS - Universit\'e de Nantes \\ Laboratoire Jean Leray \\ 2, rue de la Houssini\`ere \\
44322 Nantes cedex 3, France }
\email{frederic.bernicot@univ-nantes.fr}

\author[P. Germain]{Pierre Germain}
\address{P. Germain, Courant Institute of Mathematical Sciences \\ 251 Mercer Street \\ New York 10012-1185 NY \\ USA}
\email{pgermain@cims.nyu.edu}

\thanks{P. Germain is partially supported by NSF grant DMS-1101269, a start-up grant from the Courant Institute, and a Sloan fellowship.}
\thanks{F. Bernicot is partially supported by the ANR under the project AFoMEN no. 2011-JS01-001-01.}

\begin{abstract}
Consider a bilinear interaction between two linear dispersive waves with a generic resonant structure (roughly speaking, space and time resonant sets intersect transversally). We derive an asymptotic equivalent of the solution for data in the Schwartz class, and bilinear dispersive estimates for data in weighted Lebesgue spaces. An application to water waves with infinite depth, gravity and surface tension is also presented.
\end{abstract}
\maketitle

\begin{quote}
\footnotesize\tableofcontents
\end{quote}

\section{Introduction}

\subsection{The set up}

\label{goeland}
We focus on the physically relevant dimensions 
$$
d = 2 \; \mbox{or} \; 3,
$$
and we will be interested in the solution $u$ of
\begin{equation}
\label{equationdebase}
\left\{ \begin{array}{l} 
i \partial_t u + a(D) u = T_m(v,w) \\
i \partial_t v + b(D) v = 0 \\
i \partial_t w + c(D) w = 0
\end{array} \right.
\qquad \mbox{with} \qquad
\left\{ \begin{array}{l}
u(t=0) = 0 \\
v(t=0) = f \\
w(t=0) = g.
\end{array} \right.
\end{equation}
We need to say more about the various characters appearing in the above equations. First, 
$u$, $v$ and $w$ are complex-valued functions of $(t,x) \in \mathbb{R} \times \mathbb{R}^d$.
Next, $T_m$ is the pseudo-product operator with symbol $m \in \mathcal{C}_0^\infty$ given by
$$
T_m(f,g)(x) \overset{def}{=} \int_{\mathbb{R}^d} \int_{\mathbb{R}^d} e^{ix \xi} m(\xi,\eta) \widehat{f}(\eta) \widehat{g}(\xi-\eta)\,d\eta\,d\xi.
$$
(this bilinear operator has the effect of restricting the study to bounded frequencies). 

Finally, $a(D)$, $b(D)$ and $c(D)$ are the Fourier multipliers with (smooth) symbols $a(\xi)$, $b(\xi)$ and $c(\xi)$ respectively.
These symbols will be throughout this text assumed to be real and such that 
\begin{equation}
\label{pingouin}
\operatorname{Hess} [a](\xi), \operatorname{Hess} [b](\xi), \operatorname{Hess} [c](\xi) \qquad \mbox{are non-degenerate for any $\xi$}
\end{equation}
(this could be relaxed in several places, but for the sake of simplicity in the statements we prefer to maintain this assumption).
This ensures a dispersive behavior for the corresponding equations: denoting $S(t)$ for any of the groups $e^{ita(D)}$, $e^{itb(D)}$,
$e^{itc(D)}$, the following dispersive estimate holds:
\begin{equation}
\label{dispersive}
\left\| S(t) f \right\|_{L^{p'}} \lesssim |t|^{\frac{1}{2}-\frac{1}{p}} \|f\|_{L^p} \qquad \mbox{for $p \in [1,2]$}
\end{equation}
(provided say $f$ has compact support in frequency).

\medskip

The question that we address in this paper is:
{\bf what are the decay properties of $u$? How do they depend on the localization of $f$ and $g$?
How do they depend on the resonant structure of the problem?}

\subsection{Space-time resonances}  \subsubsection{A quick presentation}

Using Duhamel's formula, $u(t,\cdot)$ is given by the following bilinear operator $u(t,\cdot)=T_t(f,g)$ with
\begin{equation}
\label{duhamel}
T_t(f,g)(x) =  \int_0^t \int \int e^{ix\xi} e^{ita(\xi)} e^{is\phi(\xi,\eta)} m(\xi,\eta) \widehat{f}(\eta) \widehat{g}(\xi-\eta)\,d\xi\,d\eta\,ds
\end{equation}
or, to put it in a more concise form,
$$ T_t(f,g)\overset{def}{=} e^{ita(D)} \int_0^t T_{m e^{is\phi}} (f,g) ds, $$
where 
$$\phi(\xi,\eta)\overset{def}{=} -a(\xi)+b(\eta)+c(\xi-\eta).$$
The goal of this article is thus to understand the behavior for large time $t\gg 1$ and exponents $q\in[2,\infty]$ of
$\left\| T_t(f,g) \right\|_{L^q}$, with $f,g$ in weighted $L^2$ spaces.

Viewing this double integral as a stationary phase problem, it becomes clear that the sets where the phase is stationary in $s$ , respectively $\eta$,
\begin{align}
& \mathcal{T} \overset{def}{=} \{(\xi,\eta) \:\mbox{such that}\;\phi(\xi,\eta) = 0\} \qquad \mbox{(time resonant set)} \\
& \mathcal{S} \overset{def}{=} \{(\xi,\eta) \:\mbox{such that}\; \nabla_\eta \phi(\xi,\eta) = 0\} \qquad \mbox{(space resonant set)}
\end{align}
will play a crucial role. Even more important is their intersection
$$
\mathcal{R} \overset{def}{=} \mathcal{T} \cap \mathcal{S} \qquad \mbox{(space-time resonant set)}.
$$
We refer to~\cite{PG} for a more comprehensive presentation of the space-time resonance method and its application to the study of PDEs.

\subsubsection{The case of a single dispersion relation} To make the previous discussion more concrete, we first examine the case where only one dispersion relation enters the problem, and describe their space-time resonant structure. We simplify even more, by assuming (only in the present subsection) that the dispersion relation is isotropic, thus of the form
$$
\mbox{for $\xi \in \mathbb{R}^d$}, \qquad \tau(\xi) = \bar \tau(|\xi|),
$$
and satisfies $\tau(0)=0$ as well as $\bar \tau'>0$. This covers many fundamental physical cases: to name a few, Schr\"odinger ($|\xi|^2$), Airy ($|\xi|^3$), waves ($|\xi|$), and, as we will see, water-waves. The wave interacts with itself or its conjugate (which has dispersion relation $-\tau$). Applying the framework described above, this gives several possible phase functions:
$$
\phi_{\pm,\pm} \overset{def}{=} \bar \tau(|\xi|)\pm \bar \tau(|\eta|) \pm \bar \tau(|\xi-\eta|) ,
$$
whose space, time, and space-time resonant sets will be respectively denoted by $\mathcal{S}_{\pm,\pm}$, $\mathcal{T}_{\pm,\pm}$ and $\mathcal{R}_{\pm,\pm}$.

\bigskip

\noindent \underline{The case $\bar \tau''>0$ or $\bar \tau''<0$}, that is $\bar \tau$ convex or concave. It is then easy to see that
\begin{align*}
& \mathcal{R}_{++} = \mathcal{T}_{++} = \{ (\xi,\eta) = (0,0) \}, \quad \mathcal{S}_{++} = \{ \xi = 2 \eta \}\\
& \mathcal{R}_{+-}  = \mathcal{S}_{+-} =  \{ \xi = 0 \}, \\
&  \mathcal{S}_{--} = \{ \xi = 2 \eta \}.
\end{align*}
To go further, we need to distinguish between $\bar \tau''>0$ and $\bar \tau''<0$ (i.e. the concave and convex cases)
\begin{itemize}
\item If $\bar \tau$ is concave,
\begin{align*}
& \mathcal{T}_{--} = \{ \eta = 0 \mbox{ or } \xi-\eta = 0\}, \quad \mathcal{S}_{--} = \{ \xi = 2 \eta \}, \quad \mathcal{R}_{--} = \{ (0,0) \} \\
& \mathcal{T}_{+-} = \{ \xi-\eta = 0 \mbox{ or } \xi =0 \} , \quad \mathcal{S}_{+-} =  \mathcal{R}_{+-} = \{ \xi = 0 \}.
\end{align*}
\item If $\bar \tau$ is convex, there is no simple expression for the space resonant sets in general. However,
\begin{align*}
& \mathcal{S}_{--} = \{ \xi = 2 \eta \}, \quad \mathcal{R}_{--} = \{ (0,0) \} \\
& \mathcal{S}_{+-} =  \mathcal{R}_{+-} = \{ \xi = 0 \}.
\end{align*}\end{itemize}
To summarize: if $\bar \tau$ is concave or convex, space-time resonant sets are linear; this makes it (relatively) easy to analyze the decay properties for the solution of (\ref{goeland}).

\bigskip

\noindent \underline{The case $\bar \tau''=0$}, that is, $\bar \tau$ linear. This case is very degenerate: space, time, and space-time resonance reduce to colinearity conditions on $\xi$, $\eta$ and $\xi-\eta$.

\bigskip

\noindent \underline{The generic case}. First notice that the assumption~(\ref{pingouin}) is not satisfied at frequencies $\xi_0$ such that
$\bar \tau''(|\xi_0|) = 0$. In particular, at this frequency, the dispersive estimates~(\ref{dispersive}) do not hold any more. The analysis presented in this paper is only valid away from the points where $\bar \tau''$ vanishes; a different kind of analysis is then necessary. 

Without making any convexity assumption on $\bar \tau$, the time resonant sets cannot be described simply in general, but more can be said about the space-time resonant sets. Without going into the details, they are generically either of the above type (linear) or of the form
$$
\{ |\xi|=R, \; \eta = \lambda \xi \}
$$
for constants $R$ and $\lambda$. This means in particular that the space-time resonant set does not have a linear structure; this is well-known to make questions of harmonic analysis more challenging. The results proved in the present paper apply to this delicate case.

\subsubsection{The case of water waves} Water waves constitute one of the fundamental dispersive equations. The classical setting involves three parameters: the gravity $g\geq 0$, the surface tension $\sigma \geq 0$, and the water depth $h \in [0,\infty]$. We shall assume that $h=\infty$. The dispersion relation for the linearized equations around a flat free surface, the fluid having zero velocity, reads then
$$
\tau(\xi) = \bar \tau(|\xi|) \quad \mbox{with} \quad \bar \tau(s) = \sqrt{gs + \sigma s^3}.
$$
Besides its physical significance, this problem is particularly interesting since different values of $g$ and $\sigma$ yield different convexity properties of $\tau$:
\begin{itemize}
\item If $\sigma = 0$ and $g>0$, $\bar \tau$ is concave.
\item If $g=0$ and $\sigma>0$, $\bar \tau$ is convex.
\item If $g,\sigma>0$, $\tau$ is concave-convex (see Subsection \ref{subsec:alpha}) with an inflection point at $s=\lambda_0>0$ such that $\sigma \lambda_0^2=\frac{(2\sqrt{3}-3)}{3} g$.
\end{itemize}

\subsubsection{Generic space-time resonant structure}
The answer to the question raised at the end of Section~\ref{goeland} depends crucially on $\mathcal{R}$, 
and more precisely how $\mathcal{S}$ and $\mathcal{T}$ intersect: what is the dimension of $\mathcal{R}$ ? Is the intersection transverse ? etc... 
Since we cannot investigate all the possibilities, we will mostly focus on the generic situation.

First, it is possible that $\mathcal{S}$ and $\mathcal{T}$ do not intersect at all; this trivial configuration is easily understood.

If $\mathcal{R} \neq 0$, we will often (but not always) impose the following generic assumptions:

\begin{itemize}

\item[(A1)] \vsp $\operatorname{Hess}_\eta [\phi]$ is non-degenerate on $\mathcal{R}$;

\medskip

\item[(A2)] $\nabla_\xi \phi \neq 0  \;\mbox{on $\mathcal{R}$}$.

\medskip

\item[(A3)] \vsp for all $\sigma\in[0,1]$, $\operatorname{Hess}_\xi [a(\xi)+\sigma \phi(\xi,\eta)]$ or $\operatorname{Hess}_{(\xi,\eta)} [a(\xi)+\sigma \phi(\xi,\eta)]$ is non-degenerate on $\mathcal{R}$.

\end{itemize}

\begin{remark} It is easy to check that the conditions (A1) and (A2) imply that $\mathcal{R}$ is a smooth $(d-1)$-dimensional manifold. \end{remark}

The meaning of condition (A1) is easily understood: this is asking that the phase $\phi$ be non-degenerate in $\eta$ on $\mathcal{R}$ (it is already stationary in $\eta$ by definition of the space-resonant set). This idea was already introduced in the paper by Ionescu and Pausader~\cite{IPa} to
obtain global solutions for a nonlinear dispersive problem (coupled quasilinear Klein-Gordon equations).

The meaning of condition (A2) is fairly clear. Notice that we will sometimes need the stronger assumption that $(\nabla_\xi \phi)^t (\operatorname{Hess} [a])^{-1} (\nabla_\xi \phi) \neq 0$, which reduces to $\nabla \phi \neq 0$ if $a$ is convex or concave.

The meaning of condition (A3) is more technical : it allows us to localize the difficulty along two different kinds of non-stationary phases. As will be observed in the proofs, the main term is the one corresponding to $\sigma$ close to $0$, and in this situation condition (A3) is always satisfied, because of (\ref{pingouin}). We refer the reader to the Appendix (Section \ref{appendice}) where we prove that Assumption (A3) is generically satisfied, at least in the context of radial dispersive relations.

Finally, notice that the non-degeneracy conditions (A1) and (A2) also appeared in dimension 1 in the previous paper by the authors \cite{BG}.

\subsection{Previous results} The concept of space-time resonances was introduced by Germain, Masmoudi and Shatah~\cite{GMS01, GMS02} to deal with nonlinear Schr\"odinger equations. It was then successfully applied by the same authors to the three-dimensional water-waves equation~\cite{GMS1, GMS2}, see also Wu~\cite{Wu}. Very recently, related ideas were brought to bear on the more difficult one-dimensional water-waves equation by Ionescu and Pusateri~\cite{IPu} and Alazard and Delort~\cite{AD1, AD2}.

All the works that have been mentioned so far deal with linear space-time resonant sets; we now review articles which deal with situations where the space-time resonant set is of the type $\{|\xi|=R,\eta= \lambda \xi\}$ (which is generic in the isotropic setting). 
Global existence, in dimension 3, for a semilinear dispersive system with such a space-time resonant set was first proved by Germain \cite{PG1}; this was then applied to the one-fluid Euler-Maxwell equations by Germain and Masmoudi~\cite{PGNM}. The previous analysis was  improved by Ionescu and Pausader \cite{IPa}, which ultimately led Guo, Ionescu and Pausader \cite{GIP} to prove global existence for the two-fluid Euler-Maxwell equations.

To the best of our knowledge, there is no work which deals with generic space-time resonant sets in dimension 2. In dimension 1, we mention the previous work by the authors of the present paper~\cite{BG}, where the same question as here is asked and answered.

Finally, in a different direction, Masmoudi and Nakanishi~\cite{MN} initiated the study of space-time resonances through limiting equations describing solutions concentrated in Fourier around some frequencies.

\subsection{Main results and organisation of the paper}

\subsubsection{Asymptotic equivalent for the solution} In Section~\ref{sectionequivalentu}, we mainly focus on the case $d=2$ (the case $d=3$ being trivial) and we are able to derive precise informations on the asymptotic form of $u$, solution of (\ref{equationdebase}), if $f$ and $g$ decay sufficiently fast. In order to sketch the main result, define $\mathcal{R}_X$ to be the set of $X \in \mathbb{R}^2$ such that
there exists $(\xi,\eta) \in \mathcal{R}$ with $X = - \nabla a (\xi)$.  Then if (A1) and (A2) hold, 
\begin{equation}
\label{pie}
u(t,x) \mbox{ is of order } \left\{ \begin{array}{l} \frac{\operatorname{log} t}{t} \mbox{ if } \frac{x}{t} \in \mathcal{R}_X \\ \frac{1}{t} \mbox{ otherwise.} \end{array} \right.
\end{equation}
The precise statement is Theorem \ref{theoremeasymptotique}.

\subsubsection{Asymptotic equivalent for the profile} While $u$ is the physical solution of \eqref{equationdebase}, it has been recognized in the mathematical works that were cited above that a lot of the mathematical structure can be read off the {\it profile}
$$
h(t) \overset{def}{=} e^{-ita(D)} u(t).
$$
Section~\ref{sectionequivalenth} is dedicated to its study, which is best done in Fourier space. It is easy to see formally that
$$
\mbox{as $t \to \infty$}, \qquad \widehat{h}(t,\xi) \longrightarrow \widehat{h}_{\infty}(\xi) \overset{def}{=} - \int \frac{1}{i \phi(\xi,\eta)} m(\xi,\eta) \widehat{f}(\eta) \widehat{g}(\xi-\eta) \, d\eta.
$$
In Theorem \ref{theoremeasymptotiqueh}, we establish the regularity of $\widehat{h_\infty}$, and the speed of convergence of $\widehat{h}$ to it.

\subsubsection{Bilinear dispersive bounds} The aim of Section~\ref{sectionbounds} is to quantify how much decay of the data (at space infinity) is necessary to obtain decay of the solution (at time infinity). The obtained estimates are as follows:

\begin{theorem} \label{thm:main} 
Assume that $d \in \{2,3\}$, $m$ is smooth and compactly supported, (\ref{pingouin}) holds, $\mathcal{R}$ is $(d-1)$-dimensional and that furthermore on $\mathcal{R}$ (A1), (A2) and (A3) hold.
Then the solution $u$ of~(\ref{equationdebase}) satisfies for every $\epsilon>0$, 
\begin{equation}
\begin{split}
& \mbox{if $d=2$,} \quad \|u(t)\|_{L^\infty} \lesssim t^{-1} \log t \|f\|_{L^{2,\frac{3}{2}+\epsilon}} \|g\|_{L^{2,\frac{3}{2}+\epsilon}} \\
& \mbox{if $d=3$,} \quad \|u(t)\|_{L^\infty} \lesssim t^{-\frac{3}{2}} \|f\|_{L^{2,\frac{3}{2}+\epsilon}} \|g\|_{L^{2,\frac{3}{2}+\epsilon}}.
\end{split}
\end{equation}
\end{theorem}

\begin{remark}
\begin{itemize}
\item First observe that the time decay which is obtained is optimal. This is clear in dimension 3, since it coincides with the linear decay; in dimension 2, this follows from~(\ref{pie}).
\item The above statement requires $f$ and $g$ to belong to $L^{2,s}$, with $s>\frac{3}{2}$. It is not clear whether this condition can be relaxed or not. In any case, in dimension 3, this seems optimal in the following sense: $s>\frac{3}{2}$ is the condition needed to obtain optimal decay in the linear problem. Thus, when solving the Cauchy problem for a nonlinear dispersive equation, this has to be assumed on the data in order to get optimal decay in $L^\infty$.
\item It would be easy to generalize the above result to $d \geq 4$ to get
$$ \|u(t)\|_{\infty} \lesssim t^{-\frac{d}{2}} \|f\|_{L^{2,\frac{d}{2}+\epsilon}} \|g\|_{L^{2,\frac{d}{2}+\epsilon}}$$
(indeed, the higher the dimension, the stronger the decay, and thus the easier the proofs).
\item It is interesting to compare this theorem to the results in the papers by Ionescu and Pausader~\cite{IPa}, as well as Guo, Ionescu and Pausader~\cite{GIP}. A byproduct of the analysis performed in these works is to obtain, in our language, in dimension $3$ an integrable decay ($\sim \frac{1}{t^{1+\delta}}$, $\delta>0$) for data in $L^{2,1+\kappa}$, with $\kappa>0$. This result as well as the method used were sources of inspiration for the present article.
\end{itemize}
\end{remark}

\subsubsection{Application to water waves}  Finally, in Section~\ref{sectionww}, we apply the previous results to the water wave equation with gravity and surface tension and deduce information on its large-time behavior: we conjecture what decay small solutions should exhibit. This decay depends on the frequency range under consideration.

\subsection{Notations} 
\begin{itemize}
\item The Fourier transform of a function $f$ on $\mathbb{R}^d$ is defined by
$$
\mathcal{F} f (\xi) = \frac{1}{(2\pi)^{d/2}} \widehat{f}(\xi) \overset{def}{=} \int_{\mathbb{R}^d} f(x) e^{-ix\cdot \xi}\,d\xi.
$$
so that $f \mapsto \widehat{f}$ is an isometry of $L^2(\R^d)$ and
$$
f(x) = \frac{1}{(2\pi)^{d/2}}  \int_{\mathbb{R}^d} \widehat{f}(\xi) e^{ix\cdot \xi}\,dx.
$$
\item Linear Fourier multipliers are given by the formula
$$
m(D) f \overset{def}{=}  \mathcal{F}^{-1} \left[ m(\xi) \widehat{f}(\xi) \right].
$$
\item Bilinear Fourier multipliers are given by the formula
$$
T_m(f,g)(x) \overset{def}{=} \int_{\mathbb{R}^d} \int_{\mathbb{R}^d} e^{ix \xi} m(\xi,\eta) \widehat{f}(\eta) \widehat{g}(\xi-\eta)\,d\eta\,d\xi
$$
so that $T_{\frac{1}{(2\pi)^{d/2}}}(f,g) = fg$.
\item The Hessian of a (real) function on $\mathbb{R}^d$ is the matrix
$$
\operatorname{Hess} [f] \overset{def}{=}  ( \partial_i \partial_j f)_{1 \leq i,j \leq d}.
$$
\item The weighted $L^2$ space $L^{2,s}$ is given by the norm
$$
\|f\|_{L^{2,s}} \overset{def}{=} \left\| \langle x \rangle^s f(x) \right\|_{L^2}
$$
\item We write $A \lesssim B$ if there exists a constant $C$ such that $A \leq CB$ and $A \sim B$ if $A \lesssim B$ and $B \lesssim A$.
\item We write $a=O(b)$ if $|a| \lesssim |b|$.
\item Finally, $R$ is the (direct) rotation of angle $\frac{\pi}{2}$, center the origin.
\end{itemize}

\section{Asymptotic equivalent for $u$}

\label{sectionequivalentu}

\subsection{Linear vs nonlinear behavior}

We assume in this section for simplicity that $f$ and $g$ belong to the Schwartz class, though 
much weaker assumptions would suffice. We focus on $u$ the solution of (\ref{equationdebase}).

A crucial question in nonlinear PDE analysis is whether $u$ simply behaves like a linear solution of
the problem, or if genuinely nonlinear phenomena can be observed. Recall that linear solutions
obey the dispersive estimates~(\ref{dispersive}). A more precise characterization of the asymptotic
behavior of linear solutions can be obtained by the stationary phase method (see for instance~\cite{Wolff}):
$$
e^{ita(D)} F(x) = e^{i(x \xi_0 + t a(\xi_0))} e^{i\frac{\pi}{4} \sigma} 
\frac{1}{|\operatorname{det} \operatorname{Hess} [a] (\xi_0) |^{1/2}} \frac{1}{t^{d/2}} \widehat{F}(\xi_0)
+ O \left( \frac{1}{t^{\frac{d}{2}+1}} \right)
$$
where $\xi_0$ is such that
$$
\nabla a (\xi_0) \overset{def}{=} -\frac{x}{t} \qquad \mbox{and} \qquad 
\sigma \overset{def}{=} \operatorname{signature} \operatorname{Hess} [a ](\xi_0).
$$

In the absence of space or time resonances, $u$ resembles as expected a solution of the linear problem.
But what if the space-time resonant set is generic? Do the decay rates agree with~(\ref{dispersive})?
\begin{itemize}
\item In dimension 1, this was investigated in~\cite{BG}, and the answer is, emphatically, no: for the one-dimensional version of~(\ref{equationdebase}), the $L^2$ norm of $u$ grows like $\log t$ instead of being constant, and the $L^\infty$ norm decays like $\frac{1}{t^{1/4}}$ instead of $\frac{1}{\sqrt{t}}$.
\item In dimension $2$, we will only observe a logarithmic discrepancy for $p=\infty$: see Theorem~\ref{theoremeasymptotique} and its corollary below.
\item In dimension $3$, the decay rates in $L^p$ agree for the linear problem and the solution of~(\ref{equationdebase}). This is very easily proved by a stationary phase argument, and left to the reader.
\end{itemize}

Thus we will focus in the remainder of this section on the case of dimension 2.

\subsection{Main result}

\begin{theorem}
\label{theoremeasymptotique}
Assume that $d=2$, and that $f$ and $g$ in~(\ref{equationdebase}) belong to the Schwartz class $\mathcal{S}$.
Consider $u$ the solution of~(\ref{equationdebase}). 
\begin{itemize}
\item[(i)] In the absence of time resonances ($\mathcal{T} \cap \operatorname{Supp} m = \emptyset$), 
$$
u(t) = - e^{ita(D)} F + O \left( \frac{1}{t^2} \right) \quad \mbox{with} \quad F = T_{\frac{m}{i \phi}}(f,g).
$$
\item[(ii)] In the absence of space resonances ($\mathcal{S} \cap \operatorname{Supp} m = \emptyset$), for any $M>0$, $N \in \mathbb{N}$, 
$$
u(t) = e^{ita(D)} F_M + O \left( \frac{1}{M^N t} \right),
$$
where
$$
F_M = \int_0^M e^{-isa(D)} T_m(e^{isb(D)}f,e^{isc(D)} g) \,ds.
$$
Thus for $t>M$, $e^{ita(D)} F_M$ is simply the solution of 
$$\left\{ \begin{array}{ll} 
u(t=0) = 0 \vsp \\
 i\partial_t u + a(D) u = T_m(v,w) & \mbox{if $t<M$} \vsp \\ 
i\partial_t u + a(D) u = 0 & \mbox{if $t>M$}. \vsp \end{array} \right.$$
\item[(iii)] Assume that $\mathcal{R} \neq 0$ and that on $\mathcal{R}$ (A1), (A2) and (A3) hold; and furthermore that
\begin{equation}
\label{pivert}
(\nabla_\xi \phi)^t (\operatorname{Hess} [a])^{-1} (\nabla_\xi \phi) \neq 0.
\end{equation}
Define $\mathcal{R}_X$ to be the set of $\bar X \in \mathbb{R}^2$ such that $\bar X = -\nabla a(\bar \xi)$ for some $(\xi,\eta) \in \mathcal{R} \cap \operatorname{Supp} m$. Taking $\operatorname{Supp} m$ sufficiently small (which can always be achieved with the help of smooth cut-off functions), and $\delta>0$ sufficiently small, there exists a compact set $K \subset \mathbb{R}^2$ containing $\mathcal{R}_X$ such that the map
\begin{align*}
\Phi : \;& \mathcal{R}_X \times [-\delta,\delta]  \rightarrow K \\
& (\bar X,\mu)  \mapsto \bar X + \mu \nabla_\xi \phi(\bar \xi,\bar \eta),
\end{align*}
where $(\bar \xi, \bar \eta) \in \mathcal{R} \cap \operatorname{Supp} m$, and $\bar X = -\nabla a(\bar \xi) \in \mathcal{R}_X$, is a smooth diffeomorphism. Let
$$
X \overset{def}{=} \frac{x}{t}.
$$
Assume that $X \in K$ and write it in the above coordinates
$$
X = \Phi (\bar X, \mu) = \bar X + \mu \nabla_\xi \phi(\bar \xi,\bar \eta)
$$
Then
$$
u(t,x) = \left\{ \begin{array}{ll} 
e^{it Y(X)}  Z(X) \frac{\log t}{t} + O \left( \frac{1}{t} \right) & \mbox{if $|\mu| \lesssim \frac{1}{\sqrt{t}}$} \\
e^{it Y(X)} Z(X) \frac{\log \langle \mu \rangle}{t} + O \left( \frac{1}{t} \right) & \mbox{if $|\mu| \gtrsim \frac{1}{\sqrt{t}}$}
\end{array} \right.
$$
for smooth functions $Y$ and $Z$ which depend on $a$, $b$, $c$, $f$, $g$, and $m$.
Consider $X_0$ in $\mathcal{R}_X$, thus there exists $(\xi_0,\eta_0) \in \mathcal{R}$ such that $-\nabla a (\xi_0) = X_0$. If $\widehat{f}(\eta_0) \widehat{g}(\xi_0-\eta_0) m(\xi_0,\eta_0) \neq 0$, then $Y$ and $Z$ are not zero close to $X_0$.
\end{itemize}
\end{theorem}

This theorem immediately implies the following corollary.

\begin{corollary}
Under the assumptions of $(iii)$ in the above theorem, and assuming that $\widehat{f}(\eta_0) \widehat{g}(\xi_0-\eta_0) m(\xi_0,\eta_0) \neq 0$ for some $(\xi_0,\eta_0) \in \mathcal{R}$, there holds for $t$ sufficiently large
\begin{align*}
&\|u(t) \|_{L^\infty} \sim \frac{\log t}{t} \\
& \|u(t) \|_{L^p} \sim \frac{1}{t^{1-\frac{2}{p}}} \qquad \mbox{if 2 $\leq p < \infty$}.
\end{align*}
\end{corollary}

\begin{remark}
\begin{enumerate}
\item An exact formula for $Y$ and $Z$ can be obtained by examining the proof of the theorem. However, it seems too complicated to be really illuminating, and we skip it here.
\item The $\log t$ correction to the $L^\infty$ decay is particularly relevant for energy estimates. Namely, it is well-known that a decay of $\frac{1}{t}$ leads in general for a quasilinear quadratic equation to a growth of Sobolev norms $\lesssim t^{C\epsilon}$, where $\epsilon$ is the size of the data. A decay of $\frac{\log t}{t}$ on the other hand would give an upper bound growing faster than any polynomial.
\item If~(\ref{pivert}) does not hold on $\mathcal{R}$, the proof of the theorem still gives a decay $\lesssim \frac{\log t}{t}$, but to compute an equivalent further information would be needed.
\end{enumerate}
\end{remark}

\subsection{Proof of $(i)$}
It is almost trivial: it suffices to observe that one can integrate out in $s$ in the formula giving $u$ to obtain
$$
u(t) = e^{ita(D)} \int_0^t T_{me^{is\phi}} (f,g)\,ds = T_{\frac{m}{i \phi}}(e^{itb(D)} f,e^{itc(D)} g) - e^{ita(D)} T_{\frac{m}{i\phi}} (f,g)
$$
and then that the first term can be dominated in $L^\infty$ by $\sim \frac{1}{t^2}$. Of course, $\frac{m}{\phi}$ is smooth since $\phi$ does
not vanish on the support of $m$.

\subsection{Proof of $(ii)$} It is good at this point to give to our problem a more standard stationary phase formulation. Changing variables 
by setting $X = \frac{x}{t}$ and $\sigma = \frac{s}{t}$ in~(\ref{duhamel}), it appears that
$$
u(t) = t \int_0^1 \iint e^{it \psi(\xi,\eta,\sigma,X)} \widehat{f}(\eta) \widehat{g}(\xi-\eta) m(\xi,\eta) \,d\eta \,d\xi \,d\sigma
$$
where
$$
\psi(\xi,\eta,\sigma,X) \overset{def}{=} a(\xi) + \sigma \phi(\xi,\eta) + X \xi.
$$
Now split the time integral giving $u$ as follows:
$$
u(t) = t \int_0^{M/t} \dots \,d\sigma + t \int_{M/t}^\epsilon \dots \,d\sigma + t \int_\epsilon^1 \dots \,d\sigma \overset{def}{=} I + II + III,
$$
where $\epsilon>0$ is small and will be fixed shortly.
The first piece, $I$, gives $e^{ita(D)} F$ as in the statement of (ii), and the third one, $III$, can be treated as in Step 3 of the proof of Theorem \ref{thm:K2} where it is proved that the main term is the one corresponding to $\sigma$ close to $0$ (more precisely, using Assumption (A3) the term $III$ can be split into two quantities one similar to $II$ and another one which is easily bounded).
Thus it suffices to check that the second one, $II$, can be made sufficiently small.
The first step in this direction is to apply the stationary phase lemma in $\xi$ to $II$. The Hessian of $\psi$ reads
$$
\operatorname{Hess}_\xi [\psi] = \operatorname{Hess} [a] + \sigma \operatorname{Hess}_\xi [\phi].
$$
By assumption, $\operatorname{Hess} [a]$ is not degenerate, thus $\operatorname{Hess}_\xi [\psi]$ will not be either for $\sigma$ small enough;
we make sure that $\epsilon>0$ is chosen such that this is indeed the case. One can then apply the stationary phase lemma (in $\xi$) to $II$:
denoting $\Xi(\eta,\sigma,X)$ for the point where $\nabla_\xi \psi$ vanishes, one obtains
$$
II = t \int_{M/t}^\epsilon \int e^{it\psi(\Xi,\eta,\sigma,X)} 
\left[ \frac{\alpha(\eta,\sigma)}{t} + \frac{\beta(\eta,\sigma)}{t^2} + O\left( \frac{1}{t^3} \right) \right] \,d\eta \,d\sigma \overset{def}{=} II_1 + II_2 + II_3,
$$
where $\alpha$ and $\beta$ are smooth functions of $(\eta,\sigma)$. The term $II_3$ is immediately seen to be good enough
for our purposes, so it suffices to consider $II_1$ and $II_2$; these two terms can be treated similarly, thus we focus
on $II_1$. Matters now reduce to showing that $II_1$ is $O \left( \frac{1}{M^N t} \right)$.

To prove this, we integrate by parts repeatedly using the identity 
\begin{equation}
\label{IBP} 
\frac{1}{i t \partial_{\eta^j} \psi} \partial_{\eta^j} e^{it\psi} = e^{it\psi}
\end{equation}
(choosing $j$ so that $|\partial_{\eta^j} \psi| \sim |\nabla_\eta \psi|$, adding cut-off functions if need be). Observe that
$$
\nabla_\eta  \left[ \psi(\Xi,\eta,\sigma,X) \right] = \underbrace{\nabla_\xi \psi}_{= 0} \nabla_\eta \Xi + \nabla_\eta \psi(\Xi,\eta,\sigma,X)
= \sigma \nabla_\eta \phi(\Xi,\eta).
$$
By assumption, $\nabla_\eta \phi$ does not vanish on $\operatorname{Supp} m$; this implies $|\nabla_\eta  \left[ \psi(\Xi,\eta,\sigma,X) \right]| \sim \sigma$.
Therefore, $N+1$ integrations by parts using~(\ref{IBP}) give
$$
|II_1| \lesssim \int_{M/t}^\epsilon \frac{d\sigma}{(t\sigma)^{N+1}} \lesssim \frac{1}{M^N t},
$$
which is the desired result!

\subsection{Proof of $(iii)$}  \label{subsec:iii} \underline{Step 0: initial decomposition}.
As in the proof of $(ii)$, we write
\begin{equation*}
\begin{split}
u(t) & = t \int_0^1 \iint e^{it \psi(\xi,\eta,\sigma,X)} \widehat{f}(\eta) \widehat{g}(\xi-\eta) m(\xi,\eta) \,d\eta \,d\xi \,d\sigma \\
& = t \int_0^{1/t} \dots + t \int_{1/t}^\epsilon \dots + t \int_\epsilon^1 \\
& \overset{def}{=} I + II + III
\end{split}
\end{equation*}
where
$$
\psi(\xi,\eta,\sigma,X) = a(\xi) + \sigma \phi(\xi,\eta) + X \xi 
$$
and the small constant $\epsilon>0$ will be fixed in the body of the proof. The term $I$ gives a linear contribution, therefore it is $O\left( \frac{1}{t} \right)$. The term $III$ can be treated as in Step 3 of the proof of Theorem \ref{thm:K2} where it is proved that the main term is the one corresponding to $\sigma$ close to $0$.  Therefore, we only need to focus on $II$. 

\bigskip
\noindent
\underline{Step 1: stationary phase in $\xi$}. Stationary points satisfy
\begin{equation}
\label{albatros}
\nabla_\xi \psi(\xi,\eta,\sigma,X) = \nabla_\xi a(\xi) + \sigma \nabla_\xi \phi(\xi,\eta) + X = 0,
\end{equation}
whereas the Hessian of $\psi$ is given by
$$
\operatorname{Hess}_\xi [\psi] = \operatorname{Hess} [a] + \sigma \operatorname{Hess}_\xi [\phi].
$$
For $\sigma= 0$, it is non-degenerate by Assumption~(\ref{pingouin}); choosing $\epsilon$ small enough, this remains true for $\sigma<\epsilon$.
Restricting to a small enough domain in $(\xi,\eta)$, which is always possible by adding a cutoff function if necessary, the implicit function theorem implies that, for given $(\sigma, \eta,X)$ the stationary point is unique, and given by a smooth function $\Xi(\eta,\sigma,X)$. Furthermore, the stationary phase lemma gives that
\begin{equation*}
\begin{split}
II & = t \int_{1/t}^\epsilon \int e^{it \psi(\Xi,\eta,\sigma,X)} \left( \frac{\alpha(\eta,\sigma,X)}{t} 
+ O \left( \frac{1}{t^2} \right) \right) \widehat{f}(\eta) \widehat{g}(\Xi-\eta) \,d\eta \,d\sigma \\
& =\int_{1/t}^\epsilon \int e^{it \psi(\Xi,\eta,\sigma,X)} \alpha(\eta,\sigma,X)
\widehat{f}(\eta) \widehat{g}(\Xi-\eta) \,d\eta \,d\sigma + O \left( \frac{1}{t} \right)\\
\end{split}
\end{equation*}
where $\alpha$ is a smooth function (an exact expression is known, but we ignore it to keep notations under control).

\bigskip
\noindent
\underline{Step 2: stationary phase in $\eta$}. This time, we want to apply the stationary phase lemma in $\eta$. Observe that
$$ \nabla_\eta \left[ \psi(\Xi,\eta,\sigma,X) \right] = \underbrace{\nabla_\xi \psi}_{=0} \nabla_\eta \Xi(\eta,\sigma,X) + \nabla_\eta \psi(\Xi,\eta,\sigma,X) = \sigma \nabla_\eta \phi(\Xi,\eta). $$
Hence
$$ \operatorname{Hess}_\eta \left[ \psi(\Xi,\eta,\sigma,X) \right] = \sigma \nabla_\xi \nabla_\eta \phi(\Xi,\eta) \nabla_\eta \Xi + \sigma \operatorname{Hess}_\eta \phi(\Xi,\eta,\sigma,X) 
\overset{def}{=} \sigma Z(\eta,\sigma,X). $$
Now observe that if $\sigma = 0$, differentiating~(\ref{albatros}) implies that
$$
\nabla_\eta \Xi (\eta,\sigma=0,X)= 0.
$$
Therefore,
$$
Z(\eta,\sigma=0,X) = \operatorname{Hess}_\eta \phi(\Xi,\eta),
$$
which is by assumption non-degenerate. Choosing $\epsilon$ small enough, this remains true for $Z(\eta,\sigma,X)$
if $\sigma<\epsilon$. We will now denote $H(\sigma,X)$ the stationary point $\eta$ for which $\left. \nabla_\eta \left[ \psi(\Xi(\eta,\sigma,X),\eta,\sigma,X) \right] \right|_{\eta=H}=0$.
Applying the stationary phase lemma in $\eta$ gives
\begin{equation}
\label{cormoran} 
\begin{split}
II & = \int_{1/t}^\epsilon e^{it \psi(\Xi,H,\sigma,X)} \left( \frac{\beta(\sigma,X)}{\sigma t} + O \left( \frac{1}{\sigma^2 t^2} \right) 
\right) \widehat{f}(H) \widehat{g}(\Xi-H) m(\Xi,H) \,d\sigma + O \left( \frac{1}{t} \right) \\
& = \int_{1/t}^\epsilon e^{it \psi(\Xi,H,\sigma,X)} \frac{\beta(\sigma,X)}{\sigma t}\widehat{f}(H) \widehat{g}(\Xi-H) m(\Xi,H) \,d\sigma + O \left( \frac{1}{t} \right),
\end{split}
\end{equation}
where $\beta$ is a smooth function.

\bigskip
\noindent
\underline{Step 3: the sets $\mathcal{R}$ and $\mathcal{R}_X$}. We first claim that the set $\mathcal{R}$ is given
by a smooth curve in $\mathbb{R}^4$. Indeed, it is the zero set of
$$
(\xi,\eta) \mapsto F(\xi,\eta) \overset{def}{=} (\phi(\xi,\eta),\nabla_\eta \phi(\xi,\eta))
$$
and it is easy to see that, at any point of $\mathcal{R}$, the kernel of the gradient of this map is one dimensional as
soon as $\nabla_\xi \phi$ is non zero, and $\operatorname{Hess}_\eta [\phi]$ is non degenerate, which we assume.

Thus it is possible to parameterize locally this set by $s\mapsto (\bar \xi(s),\bar \eta(s))$, a smooth map from, say, $[0,1]$ to $\mathbb{R}^4$ with a non-vanishing gradient.
Under this parameterization, $\partial_s \bar \xi$ does not vanish. Indeed, on $\mathcal{R}$,
$$
0 = \partial_s F(\bar \xi,\bar \eta) = (\nabla_\xi \phi \cdot \partial_s \bar \xi,\operatorname{Hess}_\eta [\phi] \partial_s \bar \eta + \nabla_\xi \nabla_\eta \phi \partial_s \bar \xi),
$$
so that $\partial_s \bar \xi = 0$ implies $\partial_s \bar \eta = 0$, which is not possible. Notice that the above identity also implies that
$$
\partial_s \bar \xi(s) = \lambda(s) R \nabla_\xi \phi \qquad \mbox{for some $\lambda(s) \in \mathbb{R}\setminus \{0\}$}
$$
(recall that $R$ is the rotation of angle $\frac{\pi}{2}$ around the origin).

Next, let $\mathcal{R}_X$ be the image of $\mathcal{R} \cap \operatorname{Supp} m$ by the map $(\xi,\eta) \mapsto -\nabla a (\xi)$. Thus $\mathcal{R}_X$ can be parameterized by $s \mapsto \bar X(s) \overset{def}{=} -\nabla a(\bar \xi(s))$.
Observe that
$$
\partial_s \bar X(s) = - \operatorname{Hess} [a] \, \partial_s \bar \xi = - \lambda \operatorname{Hess}[ a ] \,R \nabla_\xi \phi.
$$
This implies that $\mathcal{R}_X$ is a smooth curve in $\mathbb{R}^2$, whose tangent is along $\operatorname{Hess} [a]\, R \nabla_\xi \phi$.
An equivalent characterization is of course to define $\mathcal{R}_X$ as the set of points $X \in \mathbb{R}^2$ such that 
\begin{equation*}
\left\{ \begin{array}{l} \nabla a(\xi) + X = 0 \\ \nabla_\eta \phi(\xi,\eta) = 0 \\ \phi(\xi,\eta) = 0 \end{array} \right.
\end{equation*}
for some point $(\xi,\eta) \in \mathcal{R} \cap \operatorname{Supp} m$.

\bigskip
\noindent
\underline{Step 4: new coordinates close to $\mathcal{R}_X$}. Given $X$ sufficiently close to $\mathcal{R}_X$, we can decompose it as
$$
X = \bar X(s) + \mu \nabla_\xi \phi(\bar \xi(s),\bar \eta(s)) \qquad \mbox{for unique $s$ and $\mu$}.
$$
This is possible since $\partial_s \bar X(s) = - \lambda \operatorname{Hess}[ a ] R \nabla_\xi \phi$ and $\mu \nabla_\xi \phi(\bar \xi(s),\bar \eta(s))$ are not colinear, namely
$$
(R \nabla_\xi \phi)^t \operatorname{Hess} [a] (R \nabla_\xi \phi) \neq 0.
$$
This last assertion follows from the assumption~(\ref{pivert}) that $(\nabla_\xi \phi)^t (\operatorname{Hess} [a])^{-1} (\nabla_\xi \phi) \neq 0$ and the identity, valid for a 2 by 2 symmetric, non degenerate matrix $M$:
$$
M^{-1} = \frac{1}{\operatorname{det}M} R^{-1} M R.
$$
We call $\Phi$ the corresponding application
$$
\Phi (\bar X,\mu) = \bar X + \mu \nabla_\xi \phi(\bar \xi,\bar \eta).
$$
Choosing $\epsilon$ small enough, it is a smooth diffeomorphism from $\mathcal{R}_X \times [-\epsilon,\epsilon]$ to a compact $K$.

\bigskip
\noindent
\underline{Step 5: an oscillatory integral in $\sigma$}. We now come back to Step 2 and recall that
\begin{equation}
\label{loriot}
II = \int_{1/t}^\epsilon e^{it \zeta_X(\sigma)} \frac{\beta(\sigma,X)}{\sigma t}\widehat{f}(H) \widehat{g}(\Xi-H)
m(\Xi,H) \,d\sigma + O \left( \frac{1}{t} \right).
\end{equation}
where we denoted $\zeta_X$ for the phase function
$$
\zeta_X(\sigma) \overset{def}{=} \psi(\Xi(H,\sigma,X),H(\sigma,X),\sigma,X).
$$
We would like to apply once again a stationary phase argument, thus we need to compute $\zeta_X''$ and make sure that it does not vanish.
\begin{align*}
& \zeta_X'(\sigma) = \partial_\sigma [\psi(\Xi,H,\sigma,X)] = \underbrace{\nabla_\xi \psi}_{=0} \partial_\sigma \Xi + \underbrace{\nabla_\eta \psi}_{=0} \partial_\sigma H + \partial_\sigma \psi = \phi(\Xi,H) \\
& \implies \zeta_X''(\sigma) = \partial_\sigma^2 [\psi(\Xi,H,\sigma,X)] = \nabla_\xi \phi \cdot \partial_\sigma \Xi + \underbrace{\nabla_\eta \phi}_{=0} \cdot \partial_\sigma H = \nabla_\xi \phi \cdot \partial_\sigma \Xi.
\end{align*}
Differentiating~(\ref{albatros}) at $\sigma = 0$ yields
$$
\left. \partial_\sigma \Xi \right|_{\sigma = 0} = - (\operatorname{Hess}[ a])^{-1} \nabla_\xi \phi,
$$
which gives, combined with the above, that
$$
\zeta_X''(0) = \left. \partial_\sigma^2 [\psi(\Xi,H,\sigma,X)] \right|_{\sigma=0} = - (\nabla_\xi \phi)^t (\operatorname{Hess} [a])^{-1} \nabla_\xi \phi,
$$
which is by assumption non zero; picking $\epsilon$ small enough, we obtain that $\zeta_X''(\sigma) \neq 0$ for $\sigma<\epsilon$.

Next, we need to estimate  the value of $\sigma_0$ such that $\zeta_X'(\sigma_0) = 0$. Since we already know $\zeta_X''(0)$,
it suffices to estimate $\zeta_X'(0)$. Using the coordinates which were defined in Step 4, and expanding in $\mu$,
\begin{equation*}
\begin{split}
\zeta_X'(0) & = \phi(\Xi(X,0),H(X,0)) \\
& = \underbrace{\phi(\Xi(\bar X(s),0), H(\bar X(s),0))}_{=0} + \underbrace{(\nabla_\eta \phi)^t}_{=0} \nabla_X H (X-\bar X(s)) 
+ (\nabla_\xi \phi)^t \nabla_X \Xi (X - \bar X(s)) \\
& \qquad \qquad \qquad + O((X-\bar X(s))^2) \\
& = - \mu (\nabla_\xi \phi)^t (\operatorname{Hess} [a])^{-1} \nabla_\xi \phi + O(\mu^2),
\end{split}
\end{equation*}
where we used in the last equality that $\nabla_X \Xi(\sigma=0) = - (\operatorname{Hess}[ a])^{-1}$ which follows from differentiating~\eqref{albatros} in $X$. We already saw that 
$$
\zeta_X''(0) = - (\nabla_\xi \phi)^t (\operatorname{Hess} [a])^{-1} \nabla_\xi \phi.
$$
Therefore
$$
\sigma_0 = - \mu + O (\mu^2),
$$
and
$$
\alpha \overset{def}{=} \operatorname{sign}(-\sigma_0) \sqrt{|\zeta(0)-\zeta(\sigma_0)|} = C_0 \mu + O(\mu^2) 
$$
for a constant $C_0$ depending on $\zeta$.

The proof of the theorem follows now from applying the following lemma to~(\ref{loriot}).

\begin{lemma}
\label{chardonneret}
Assume that $\chi$ is a smooth, compactly supported function, and $\zeta$ a smooth function such that $\zeta'''$ is bounded from above, $\zeta''$ is bounded from below and $\zeta'$ is vanishing 
at $\sigma_0$:
$$
\zeta'' \geq c > 0  \quad \mbox{and} \quad \zeta'(\sigma_0) = 0.
$$
Then
$$
\int_{1/t}^\infty \frac{e^{it\zeta(\sigma)}}{\sigma} \chi(\sigma)\,d\sigma = 
\left\{ \begin{array}{ll} \chi(0) e^{it\zeta(0)} \log t + O(1) & \mbox{if $|\alpha|\lesssim \frac{1}{\sqrt t}$} \\
\chi(0) e^{it\zeta(0)} F(\alpha) + O(1) & \mbox{if $\alpha \gtrsim \frac{1}{\sqrt t}$} \\
\chi(0) e^{-it\zeta(0)} \bar F(- \alpha) + O(1) & \mbox{if $\alpha \lesssim - \frac{1}{\sqrt t}$} \\
\end{array} \right.
$$
where
$$
F(z) \overset{def}{=} \int_z^\infty \frac{e^{i\tau}}{\tau}\,d\tau
$$
is such that $F(z) = - \log z + O(1)$, and
$$
 \alpha \overset{def}{=} \operatorname{sign}(-\sigma_0) \sqrt{\zeta(0)-\zeta(\sigma_0)}.
$$
\end{lemma}

\subsection{Proof of Lemma~\ref{chardonneret}} 
\label{macareux}
\underline{Step 1: reducing matters to a model case}.
Let us assume for a moment that we have the following estimate for a model case: fix $\chi \in \mathcal{C}^\infty_0$, then if $0<\epsilon << \frac{1}{\sqrt{t}}$ and $|\alpha| \leq 1$, 
\begin{equation}
\label{colibri}
\int_\epsilon^\infty \frac{e^{it(\alpha+\sigma)^2}}{\sigma} \chi(\sigma) \,d\sigma = \left\{ \begin{array}{ll} - \chi(0) e^{it\alpha^2} \log (\sqrt{t}\epsilon) + O(1) & \mbox{if $|\alpha|\lesssim \frac{1}{\sqrt{t}}$} \\
\chi(0) e^{it\alpha^2} F(2 \epsilon t \alpha) + O(1) & \mbox{if $\alpha\gtrsim \frac{1}{\sqrt{t}}$}\\
\chi(0) e^{-it\alpha^2} \bar F(- 2 \epsilon t \alpha) + O(1) & \mbox{if $\alpha \lesssim - \frac{1}{\sqrt{t}}$}                                                                                   \end{array}
\right.
\end{equation}
and let us show that the lemma follows. First it is possible to assume that $\zeta(\sigma_0)=0$. Next, define the smooth function
$$
\varphi(\sigma) = \operatorname{sign} (\sigma-\sigma_0) \sqrt{\zeta(\sigma)} + \sigma_0
$$
and use it to change the integration variable to $\tau = \varphi(\sigma)$. This gives
\begin{equation}
\label{pingouin1}
\int_{1/t}^\infty \frac{e^{it\zeta(\sigma)}}{\sigma} \chi(\sigma)\,d\sigma = \int_{\varphi(1/t)}^\infty \frac{e^{it(\tau-\sigma_0)^2}}{\varphi^{-1}(\tau)} \chi(\varphi^{-1}(\tau)) (\varphi^{-1})'(\tau) \,d\tau.
\end{equation}
Now observe that $\varphi^{-1} (\tau) = 0$ if $\tau = \tau_0 = \varphi(0) = \operatorname{sign}(-\sigma_0) \sqrt{\zeta(0)} + \sigma_0$. Therefore, one can write 
$\frac{1}{\varphi^{-1}(\tau)} = \frac{1}{\tau - \tau_0} \gamma(\tau)$ for a smooth function $\gamma(\tau)$ such that
$\gamma(\tau_0) = \frac{1}{(\varphi^{-1})'(\tau_0)}$. Still denoting $\gamma(\tau)$ for a smooth function, different from the previous one,
the above can be written
\begin{equation}
\label{pingouin2}
\int_{1/t}^\infty \frac{e^{it\zeta(\sigma)}}{\sigma} \chi(\sigma)\,d\sigma  = \int_{\varphi(1/t)}^\infty \frac{e^{it(\tau-\sigma_0)^2}}{\tau-\tau_0} \gamma(\tau) \,d\tau,
\end{equation}
with $\gamma(\tau_0) = \chi(0)$. A last change of variables gives, still for another smooth function $\gamma$ such that $\gamma(0) = \chi(0)$,
$$
\int_{1/t}^\infty \frac{e^{it\zeta(\sigma)}}{\sigma} \chi(\sigma)\,d\sigma  = \int_{\varphi(1/t)-\varphi(0)}^\infty \frac{e^{it(\tau + \tau_0 -\sigma_0)^2}}{\tau} \gamma(\tau)\,d\tau.
$$
This has the desired form appearing in~(\ref{colibri}) upon setting
\begin{align*}
& \alpha = \tau_0 -\sigma_0 = \operatorname{sign} (-\sigma_0) \sqrt{\zeta(0)} \\
& \epsilon =  \varphi(1/t)-\varphi(0) = \frac{|\zeta'(0)|}{2 \sqrt{\zeta(0)}} \frac{1}{t}
+O \left(\frac{1}{t^2} \right)
\end{align*}
(this last identity can be easily checked by distinguishing the cases $\sigma_0 < 0$, $0< \sigma_0 < \frac{1}{t}$, and $\sigma_0 > \frac{1}{t}$).

\bigskip
\noindent
\underline{Step 2: proof of the model case estimate~\ref{colibri} if $ \sqrt{t} |\alpha| \lesssim 1$}. 
Observing first that the smooth function $\chi$ can be removed if one multiplies the expression by $\chi(0)$, and performing afterwards a change of integration variable, one obtains
\begin{equation}
\label{petrel}
\begin{split}
\int_{\epsilon}^\infty \frac{e^{it(\alpha+\sigma)^2}}{\sigma} \chi(\sigma) \,d\sigma & = \chi(0) e^{it\alpha^2} \int_{\epsilon}^\infty \frac{e^{it(\sigma^2 + 2\alpha\sigma)}}{\sigma} \,d\sigma +O(1) \\
& = \chi(0) e^{it\alpha^2} \int_{\sqrt{t} \epsilon}^\infty \frac{e^{i(s^2 + 2ys)}}{s}\,ds + O(1), 
\end{split}
\end{equation}
where $s = \sqrt{t} \sigma$ and $y = \sqrt{t}\alpha$.
We now introduce a cut-off function $\mu$ which is smooth, equal to $1$ on $B(0,1)$, and $0$ outside of $B(0,2)$. It is easy to see that
\begin{equation*}\begin{split}
\int_{\sqrt{t} \epsilon}^\infty \frac{e^{i(s^2 + 2ys)}}{s}\,ds & = \int_{\sqrt{t} \epsilon}^\infty \frac{e^{i(s^2 + 2ys)}}{s}\mu(s) \,ds + O(1) \\
& = \int_{\sqrt{t} \epsilon}^\infty \frac{e^{i(s^2 + 2ys)}-1}{s}\mu(s) \,ds + \int_{\sqrt{t} \epsilon}^\infty \frac{1}{s}\mu(s) \,ds.
\end{split}
\end{equation*}
A straightforward estimate gives on the one hand
$$
\left| \int_{\sqrt{t} \epsilon}^\infty \frac{e^{i(s^2 + ys)}-1}{s}\mu(s) \,ds \right| \lesssim \left| \int_{\sqrt{t} \epsilon}^\infty (s + |y|) \mu(s) \,ds \right| = O(1),
$$
and on the other hand
$$
\int_{\sqrt{t} \epsilon}^\infty \frac{1}{s}\mu(s) \,ds =- \log(\sqrt{t} \epsilon) + O(1).
$$
Therefore, coming back to~(\ref{petrel}), we get
$$
\int_{\epsilon}^\infty \frac{e^{it(\alpha+\sigma)^2}}{\sigma} \chi(\sigma) \,d\sigma = -\chi(0) e^{it\alpha^2} \log(\sqrt{t} \epsilon) + O(1),
$$
which is the desired estimate.

\bigskip
\noindent
\underline{Step 3: proof of the model case estimate~\ref{colibri} if $\sqrt{t}|\alpha| \gtrsim 1$}. We only deal with $\alpha \geq 0$ for simplicity. Proceeding as in step 2, we obtain
\begin{equation*}
\int_{\epsilon}^\infty \frac{e^{it(\alpha+\sigma)^2}}{\sigma} \chi(\sigma) \,d\sigma 
 = \chi(0) e^{it\alpha^2} \int_{\sqrt{t} \epsilon}^\infty \frac{e^{i(s^2 + 2ys)}-e^{2iys}}{s}\mu(s) \,ds + \chi(0) \int_{\sqrt{t} \epsilon}^\infty \frac{e^{2iys}}{s}\mu(s) \,ds + O(1).
\end{equation*}
Now observe that on the one hand
$$
\left| \int_{\sqrt{t} \epsilon}^\infty \frac{e^{i(s^2 + 2ys)}-e^{i2ys}}{s}\mu(s) \,ds \right| \lesssim \int_{\sqrt{t} \epsilon}^\infty s |\mu(s)|\,ds = O(1)
$$
whereas on the other hand, since $|y|>1$,
$$
\int_{\sqrt{t} \epsilon}^\infty \frac{e^{2iys}}{s}\mu(s) \,ds = \int_{\sqrt{t} \epsilon}^\infty \frac{e^{2iys}}{s}\,ds + O(1)
= \int_{2 y \sqrt{t} \epsilon}^\infty \frac{e^{iv}}{v}\,dv + O(1) = F (2 y\sqrt{t}\epsilon) + O(1).
$$
Combining the two previous equalities gives the desired result:
$$
\int_{\epsilon}^\infty \frac{e^{it(\alpha+\sigma)^2}}{\sigma} \chi(\sigma) \,d\sigma = \chi(0) e^{it\alpha^2} F (2 y\sqrt{t}\epsilon) + O(1).
$$

\section{Asymptotic equivalent for $h=e^{-ita(D)}u$}
\label{sectionequivalenth}

\subsection{Main result}

Recall that
$$
h = e^{-ita(D)} u(t)
$$
so that
$$
\widehat{h}(t,\xi) = \int_0^t \int e^{is\phi(\xi,\eta)} m(\xi,\eta) \widehat{f}(\eta) \widehat{g}(\xi-\eta)\,d\eta\,ds.
$$
A formal computation gives easily that, under generic assumptions (excluding $\phi$ constant for instance), $\widehat{h}$ approaches
$$
\widehat{h_\infty}(\xi) = - \int \frac{1}{i \phi(\xi,\eta)} m(\xi,\eta) \widehat{f}(\eta) \widehat{g}(\xi-\eta)\,d\eta.
$$
as $t \rightarrow \infty$. For $f$ and $g$ in the Schwartz class, the following theorem investigates the smoothness of $\widehat{h_\infty}$ (and consequently, the decay of $h$), and how fast this limiting profile is approached.

\begin{theorem} 
\label{theoremeasymptotiqueh} 
\label{mesange}
Assume that $f$ and $g$ in~(\ref{equationdebase}) belong to the Schwartz class $\mathcal{S}$.
\begin{itemize}

\item[(i)] For $d=2,3$, in the absence of time resonances ($\mathcal{T} \cap \operatorname{Supp} m = \{ 0 \}$), the profile $\widehat{h_\infty}$ is smooth. If (A1) holds,
$$
\left\| \widehat{h}(t,\xi) - \widehat{h_\infty}(\xi) \right\|_\infty \lesssim \frac{1}{t^{d/2}}.
$$

\item[(ii)] For $d=2,3$, in the absence of space resonances ($\mathcal{S} \cap \operatorname{Supp} m = \{ 0 \}$), the profile $\widehat{h_\infty}$ is smooth. If (A1) holds,
$$
\left\| \widehat{h}(t,\xi) - \widehat{h_\infty}(\xi) \right\|_\infty \lesssim \frac{1}{t^N} \qquad \mbox{for any $N$}.
$$

\item[(iii)] If $d=2$, and space-time resonances occur ($\mathcal{R}  \cap \operatorname{Supp} m \neq \{ 0 \}$), the profile $\widehat{h_\infty}$ is smooth except on outcome frequencies, corresponding to 
the $\xi$-projection of the space-time resonant set:
$$
\mathcal{O} \overset{def}{=} \pi_\xi \mathcal{R}.
$$
To describe more precisely this singularity, let, for any given $\xi$, $H(\xi)$ be such that $(\xi,H(\xi)) \in \mathcal{S}$. Then, as $\xi$ approaches $\mathcal{O}$,
\begin{equation}
\label{pinson2}
\widehat{h_\infty}(\xi) \sim \widehat{f}(H(\xi)) \widehat{g}(\xi-H(\xi)) m(\xi,H(\xi))\beta(\xi,\eta) \log |\phi(\xi,H(\xi))|,
\end{equation}
where $\beta$ is a smooth function depending only on $\phi$. Furthermore,
\begin{equation}
\label{pinson1}
\widehat{h}(t,\xi) = \widehat{f}(H(\xi)) \widehat{g}(\xi-H(\xi)) m(\xi,H(\xi)) \beta(\xi,\eta) Z(t,\phi(\xi,H(\xi))) + O(1)
\end{equation}
where
$$
Z(t,u) = \int_{1/t}^1 e^{it \sigma u} \frac{d\sigma}{\sigma} = 
\left\{ \begin{array}{l} \log t + O(tu) \;\;\;\mbox{if $u<<\frac{1}{t}$} \\ \log u + O(1) \;\;\;\mbox{if $u>>\frac{1}{t}$}. \end{array} \right. 
$$

\item[(iv)] If $d=3$, space-time resonances occur ($\mathcal{R}  \cap \operatorname{Supp} m \neq \{ 0 \}$), and (A1) holds, the profile $\widehat{h_\infty}$ is approached at the rate
$$
\left\| \widehat{h}(t,\xi) - \widehat{h_\infty}(\xi) \right\|_\infty \lesssim \frac{1}{\sqrt{t}}.
$$
It is smooth, except on $\mathcal{O}$, where it is only $\mathcal{C}^{1/2-\epsilon}$ for any $\epsilon$. To describe more precisely this singularity, we consider the derivative of $\widehat{h_\infty}$. It satisfies, as $\xi$ approaches $\mathcal{O}$,
\begin{equation}
\label{moineau2}
\nabla_\xi \widehat{h_\infty} (\xi) \sim i \nabla_\xi \phi(\xi,H(\xi)) \widehat{f}(H(\xi)) \widehat{g}(\xi-H(\xi)) m(\xi,H(\xi))\beta(\xi,\eta) \frac{C_{\operatorname{sign} (\phi(\xi,H(\xi))}}{\sqrt{ |\phi(\xi,H(\xi))|}},
\end{equation}
where $C_- = \overline{C_+}$ is a complex-valued constant.
Furthermore,
\begin{equation}
\label{moineau1}
\nabla_\xi \widehat{h}(t,\xi) = \sqrt{t} i \nabla_\xi \phi(\xi,H(\xi)) \widehat{f}(H(\xi)) \widehat{g}(\xi-H(\xi)) m(\xi,H(\xi))\beta(\xi,\eta) Y(t \phi(\xi,H(\xi))
\end{equation}
where
$$
Y(u) \overset{def}{=} \int_0^1 \frac{e^{i\sigma u}}{\sqrt{\sigma}} \,d\sigma = \left\{ \begin{array}{l} 2 + O(u) \qquad \mbox{as $u \to 0$} \\ \frac{C_{\pm}}{\sqrt{u}} \qquad \mbox{as $u \to \pm \infty$}. \end{array} \right.
$$
\end{itemize}
\end{theorem}

The theorem above gives a precise description of the singularity of $\widehat{h_\infty}$ which appears in the case where $\mathcal{R}$ is not empty. It is roughly speaking of the form
\begin{align*}
& \log |\phi(\xi,H(\xi))| \qquad \mbox{if $d=2$} \\
& |\phi(\xi,H(\xi))|^{1/2} \qquad \mbox{if $d=3$}.
\end{align*}
The function $\phi(\xi,H(\xi))$ vanishes on $\mathcal{O}$, and its derivative for $\xi \in \mathcal{O}$ equals
$$
\nabla_\xi \left[  \phi(\xi,H(\xi)) \right] = \nabla_\xi \phi + \underbrace{\nabla_\eta \phi}_{=0} \nabla_\xi H(\xi) = \nabla_\xi \phi.
$$
Under the assumption (A2), $\nabla_\xi \phi \neq 0$. As is easily seen, this first implies that $\mathcal{O}$ is a smooth $(d-1)$-manifold. Second, the singularity can be thought of as being of the form
\begin{align*}
\left\{ \begin{array}{l} 
\log | \operatorname{dist}(\xi,\mathcal{O}) | \qquad \mbox{if $d=2$} \\
 | \operatorname{dist}(\xi,\mathcal{O}) |^{1/2} \qquad \mbox{if $d=3$}.         
        \end{array}
\right.
\end{align*}
This singularity will determine the decay at space infinity of $h$. Before proceeding, we make another generic assumption, namely that $\mathcal{O}$ has a non-vanishing Gauss curvature. It is then a classical result of harmonic analysis that the above singularities lead to decays at infinity of the type
\begin{align*}
& |h_\infty(x)| \sim \frac{1}{|x|^{3/2}} \qquad \mbox{if $d = 2$} \\
& |h_\infty(x)| \sim \frac{1}{|x|^{5/2}} \qquad \mbox{if $d=3$}
\end{align*}
(more precisely, such a decay is observed in directions normal to the singular surface). This implies that
\begin{align*}
h_\infty(x) \in L^p \quad \mbox{if and only if} \quad \left\{ \begin{array}{l} p > \frac{4}{3} \quad \mbox{if $d=2$} \\ p > \frac{5}{6} \quad \mbox{if $d=3$}
\end{array} \right.
\end{align*}
This has the following implication, which should be particularly relevant for global well-posedness problems: {\it the decay of $u$ as $t \to \infty$ is actually better than the one suggested by the dispersive estimates~(\ref{dispersive}) and the Lebesgue spaces to which $h_\infty$ belongs}. For instance, in dimension 3, we saw in Section~\ref{sectionequivalentu} that $u$ decays like $t^{-3/2}$ as $t \to \infty$. Heuristically, this would follow from the dispersive estimates~\eqref{dispersive} and the fact that $h_\infty$ belongs to $L^1$... except that $h_\infty \notin L^1$! Thus the optimal decay cannot be read off from knowledge of the profile and $L^p - L^q$ dispersive inequalities; Klainerman-Sobolev inequalities instead of $L^p - L^q$ dispersive inequalities would not give the optimal decay either.

\subsection{Proof of Theorem~\ref{theoremeasymptotiqueh} } 

\noindent \underline{Proof of (i).}  The smoothness of $\widehat{h_\infty}$ is obvious. To prove the convergence at the right rate, observe that
$$
\widehat{h}(t,\xi) - \widehat{h_\infty}(\xi) =  \int \frac{e^{it\phi(\xi,\eta)}}{i \phi(\xi,\eta)} m(\xi,\eta) \widehat{f}(\eta) \widehat{g}(\xi-\eta)\,d\eta.
$$
The stationary phase lemma gives the desired decay for the right-hand side.

\bigskip

\noindent \underline{Proof of (ii).} Repeated integrations by parts in $\eta$ show that
$$
\left| \nabla_{\xi}^k \int e^{is\phi(\xi,\eta)} m(\xi,\eta) \widehat{f}(\eta) \widehat{g}(\xi-\eta)\,d\eta \right| \lesssim s^{-N}
$$ 
for any $N$. This immediately gives the desired result.

\bigskip

\noindent \underline{Proof of (iii).} 
The profile $\widehat{h}$ is given by
\begin{equation*}
\begin{split}
\widehat{h}(t,\xi) & = \int_0^t \int e^{is\phi(\xi,\eta)} \widehat{f}(\eta) \widehat{g}(\xi-\eta) m(\xi,\eta) \,d\eta\,ds \\
& = t \int_0^1 \int e^{it \sigma \phi(\xi,\eta)} \widehat{f}(\eta) \widehat{g}(\xi-\eta) m(\xi,\eta) \,d\eta\,d\sigma \\
& = t \int_{0}^{1/t} \dots + t \int_{1/t}^1 \int e^{it \sigma \phi(\xi,\eta)} \widehat{f}(\eta) \widehat{g}(\xi-\eta) m(\xi,\eta) \,d\eta\,d\sigma \\
\end{split}
\end{equation*}
The first term in the last line above is obviously $O(1)$. We turn to the second, which can be analysed by the stationary phase theorem (in $\eta$)
\begin{align*}
\lefteqn{t \int_{1/t}^1 \int e^{it\sigma \phi(\xi,\eta)} \widehat{f}(\eta) \widehat{g}(\xi-\eta) m(\xi,\eta) \,d\eta\,d\sigma} & & \\ 
& & = t \int_{1/t}^1 \frac{1}{t\sigma} \beta(\xi,\eta) e^{it\sigma \phi(\xi,H(\xi))} \widehat{f}(H(\xi)) \widehat{g}(\xi-H(\xi)) m(\xi,H(\xi))\,d\sigma + O(1).
\end{align*}
for a smooth function $\beta$. This is exactly~(\ref{pinson1}), from which~(\ref{pinson2}) follows.

\bigskip

\noindent \underline{Proof of (iv).} The first assertion simply follows from the formulas giving $\widehat{h}$ and $\widehat{h_\infty}$ and the observation that, since (A1) holds, the stationary phase lemma gives
$$
\left| \int e^{is\phi(\xi,\eta)} \widehat{f}(\eta) \widehat{g}(\xi-\eta) m(\xi,\eta) \,d\eta \right| \lesssim \frac{1}{s^{3/2}}.
$$
Next, we turn to the derivative of $\widehat{h}$. It is given by
\begin{align*}
\nabla_\xi \widehat{h}(t,\xi) = & \int_0^t i s \nabla_\xi \phi(\xi,\eta) e^{is\phi(\xi,\eta)} m(\xi,\eta) \widehat{f}(\eta) \widehat{g}(\xi-\eta) \,d\eta \,ds \\
& \qquad \qquad + \int_0^t e^{is\phi(\xi,\eta)} \nabla_\xi [m(\xi,\eta) \widehat{f}(\eta) \widehat{g}(\xi-\eta)] \,d\eta \,ds.
\end{align*}
The second term in the above right-hand side is easily seen to be $O(1)$. Estimating the first one by stationary phase, and recalling that $H(\xi)$ is such that $(\xi,H(\xi)) \in \mathcal{S}$, gives
$$
\nabla_\xi \widehat{h}(t,\xi) = \int_0^t i \nabla_\xi \phi(\xi,H(\xi)) e^{is\phi(\xi,H(\xi))} m(\xi,H(\xi)) \widehat{f}(H(\xi)) \widehat{g}(\xi-H(\xi)) \frac{\beta(\xi,\eta)}{\sqrt{s}} \,ds + O(1)
$$
for a smooth function $\beta$. Changing variables to $\sigma = \frac{s}{t}$, this becomes
$$
\nabla_\xi \widehat{h}(t,\xi) = \sqrt{t} i  \nabla_\xi \phi(\xi,H(\xi)) m(\xi,H(\xi)) \widehat{f}(H(\xi)) \widehat{g}(\xi-H(\xi)) \beta(\xi,H(\xi)) \int_0^1  \frac{e^{it \sigma \phi(\xi,H(\xi))}}{\sqrt{\sigma}} \,d\sigma.
$$
This is exactly~(\ref{moineau1}), from which~(\ref{moineau2}) follows.

\section{Bilinear dispersive bounds}

\label{sectionbounds}

\label{bounds}

\subsection{Statement of the results}

We want to quantify the decay of the data required to prove the optimal decay estimates of Theorem \ref{theoremeasymptotique}.

\begin{theorem} \label{thm:d2}
Assume that $d=2$ and the assumptions (A1), (A2) and (A3) are satisfied. For initial data $(f,g)$, consider $u$ the solution of~(\ref{equationdebase}). 
Then for every $\epsilon>0$, 
$$ \|u(t)\|_{L^\infty} \lesssim t^{-1} \log(t) \|f\|_{L^{2,\frac{3}{2}+\epsilon}} \|g\|_{L^{2,\frac{3}{2}+\epsilon}},$$
and the time-decay is optimal.
\end{theorem}

\begin{remark}
We prove a slightly more precise result: for every $\epsilon>0$ we have
\begin{align*}
\|u(t)\|_{L^\infty} \lesssim & t^{-1} \left[ \|f\|_{L^{2,\frac{3}{2}+\epsilon}} \|g\|_{L^{2,\frac{1}{2}+\epsilon}} + \|f\|_{L^{2,\frac{1}{2}+\epsilon}}\|g\|_{L^{2,\frac{3}{2}+\epsilon}} \right] \\
& + t^{-1}\log(t) \|f\|_{L^{2,1+\epsilon}} \|g\|_{L^{2,1+\epsilon}}.  
\end{align*}
\end{remark}

\begin{corollary} \label{cor:d2} Under the same assumptions, for $\sigma\in(0,1)$ and every $s>\frac{3}{2}\sigma$, we have
$$ \|u(t)\|_{L^\infty} \lesssim t^{-\sigma}  \|f\|_{L^{2,s}} \|g\|_{L^{2,s}}.$$
\end{corollary}

\begin{proof} The statement is obtained by interpolation between Theorem \ref{thm:d2} and the estimate
$$ \|u(t)\|_{L^\infty} \lesssim \log(t)  \|f\|_{L^{2}} \|g\|_{L^{2}},$$
which can be easily obtained as follows
\begin{align*}
 \|u(t)\|_{L^\infty} & \lesssim \int_0^{t-1} (t-s)^{-1} \|T_{m}\|_{L^2 \times L^2 \to L^1} \|e^{isb(D)}f\|_{L^2} \|e^{isc(D)}g\|_{L^2} \, ds \\
 & \hspace{2cm} + \int_{t-1}^1 \|T_{m}\|_{L^2 \times L^2 \to L^1} \|e^{isb(D)}f\|_{L^2} \|e^{isc(D)}g\|_{L^2} \, ds \\
 & \lesssim  \log(t) \|f\|_{L^2}\|g\|_{L^2},
\end{align*}
where we used the dispersive estimates for $s\in(0,t-1)$ and Bernstein inequalities for $s\in(t-1,t)$.
\end{proof}

\begin{theorem} \label{thm:d3}
Assume that $d=3$ and that the assumptions (A1), (A2) and (A3) are satisfied. For initial data $(f,g)$, consider $u$ the solution of~(\ref{equationdebase}). 
Then for every $\epsilon>0$, 
$$ \|u(t)\|_{L^\infty} \lesssim t^{-\frac{3}{2}} \|f\|_{L^{2,\frac{3}{2}+\epsilon}} \|g\|_{L^{2,\frac{3}{2}+\epsilon}},$$
and the time-decay is optimal (since it corresponds to the decay for the linear solution).
\end{theorem}

As for Corollary \ref{cor:d2}, we deduce the following:

\begin{corollary} \label{cor:d3} Under the same assumptions, for $\sigma\in[0,\frac{3}{2}]$ and every $s>\sigma$, we have
$$ \|u(t)\|_{L^\infty} \lesssim t^{-\sigma}  \|f\|_{L^{2,s}} \|g\|_{L^{2,s}}.$$
\end{corollary}

\begin{proof}[Proof of theorems \ref{thm:d2} and \ref{thm:d3}]
Recall that the time-resonant set is given by ${\mathcal T}=\{(\xi,\eta), \phi(\xi,\eta)=0\}$, the space-resonant set by ${\mathcal S}=\{(\xi,\eta), \nabla_\eta \phi(\xi,\eta)=0\}$, whereas the space-time resonant set is the intersection of these two $\mathcal{R}= {\mathcal S} \cap {\mathcal T}$.

Define the subsets $\Omega_1$ and $\Omega_2$ of $\mathbb{R}^d \times \mathbb{R}^d$ by
\begin{itemize}
\item $\Omega_1$ is the set where $2 |\phi|\geq |\nabla \phi|$;
\item $\Omega_2$ is the set where $|\phi|\leq 2 |\nabla \phi|$.
\end{itemize}
The set $\Omega_1$ can be thought of as a truncated ``cone'' around the sub-manifold ${\mathcal S}$ with a top given by ${\mathcal R}$; it contains the space-resonances. Similarly for $\Omega_2$ around the sub-manifold ${\mathcal T}$, which contains the time-resonances. From this decomposition, we define two symbols $m_1$ and $m_2$ such that
\begin{itemize}
\item $m_1$, $m_2$ are non-negative and smooth away from $\mathcal{R}$, and add up to $m$;
\item $m_1 = 1$ on $\Omega_2^\complement$ and $m_1 = 0$ on $\Omega_1^\complement$.
\end{itemize}
Finally, we choose $m_1$ and $m_2$ so that they enjoy natural derivative bounds, namely
$$
|\partial_\xi^\alpha \partial_\xi^\eta m_i(\xi,\eta)| \lesssim \frac{1}{\operatorname{dist}((\xi,\eta),\mathcal{R})^{|\alpha|+|\beta|}} \quad \mbox{for $i=1,2$, any $\alpha$ and $\beta$}.
$$
This decomposition gives rise from the smooth symbol $m$ to two symbols $m_i$ and we have
$$ u(t)=u_1(t)+u_2(t)$$
with
$$ \widehat{u_i}(t,\xi) \overset{def}{=}  e^{ita(\xi)} \int_0^t \int e^{is\phi(\xi,\eta)} m_i(\xi,\eta) \widehat{f}(\eta) \widehat{g}(\xi-\eta) \,d\eta\,ds.$$
The next subsections are devoted to separately estimating these two functions. By combining these results (in the case $d=2$, see Proposition \ref{prop:u1} and Theorem \ref{thm:K2} and in the case $d=3$, Proposition \ref{prop:u1} and Theorem \ref{thm:K3}), the theorems follow.
\end{proof}

\subsection{Study of $u_1$ corresponding to space-resonances}

\begin{proposition}[Estimate of $u_1$] \label{prop:u1} Let $d\in\{2,3\}$, then for every $s>\frac{d}{2}$
\begin{equation} \label{eq:u1} \|u_1(t)\|_{L^\infty} \lesssim t^{-\frac{d}{2}}  \|f\|_{L^{2,s}} \|g\|_{L^{2,s}}. \end{equation}
\end{proposition}

\begin{proof}
Integrating out in the Duhamel formula, $u_1$ can be written
$$ u_1(t) = I_t(f,g)-II_t(f,g),$$
with
$$ I_t(f,g) = T_{m_1/\phi}(e^{itb(D)} f, e^{itc(D)}g)$$
and
$$ II_t = e^{ita(D)} T_{m_1/\phi}(f,g).$$
Split the domain $\Omega_1$ into almost disjoint regions
$$ 
\Omega_1 = \cup_j \left[ \Omega_1\cap \{ 2^{j-1} \leq |\phi| < 2^{j+2}\}\right]\overset{def}{=}\cup_j \Omega^j_1 .
$$
We can assume that the union only runs over $j \leq 0$: $\Omega = \cup_{j\leq 0} \Omega^j_1$. Since $\Omega_1$ is a truncated ``cone'' around the sub-manifold ${\mathcal S}$ with top the manifold ${\mathcal R}$ of dimension $d-1$, the  region $\Omega_1^j$ has a volume of order
\begin{equation}
|\Omega_1^j| \lesssim 2^{j(d+1)}. \label{eq:vol1}
\end{equation}
Next, associate to the partition $\Omega^j_1$ a smooth partition of unity $\chi_{\Omega^j_1}$ in a natural way: we require that 
$$
\sum_{j\leq 0} \chi_{\Omega^j_1}=1 \;\;\mbox{on} \;\; \Omega_1, \quad \operatorname{Supp} \chi_{\Omega_1^j} \subset \Omega_1^{j}, \quad \mbox{and} \quad |\partial_\xi^\alpha \partial_\eta^\beta \chi_{\Omega_1^j}(\xi,\eta)| \lesssim 2^{-j(|\alpha|+|\beta|)}
$$  
for any $\alpha$ and $\beta$.
By defining $E_j\subset \R^d$ the projection of $\Omega^j_1$ on the $(\xi-\eta)$-variable and $F_j$ on the $\eta$-variable, we have
$$ |E_j|+|F_j| \lesssim 2^{j}.$$
By the dispersive estimates (\ref{dispersive})
\begin{align*}
 \|I_t(f,g)\|_{L^\infty} & \lesssim \sum_{j \leq 0}  \left\| T_{\chi_{\Omega^j_1} m_1/\phi} (e^{itb(D)} f, e^{itc(D)}g) \right\|_{L^\infty} \\
 & \lesssim \sum_{j \leq 0} \left\| T_{\chi_{\Omega^j_1} m_1/\phi}  \right\|_{L^{2} \times L^{\infty} \to L^\infty} t^{-\frac{d}{2}} \|P_{F_j} f\|_{L^{2}} \| g \|_{L^{1}},
\end{align*}
where $P_{F_j}$ is the Fourier multiplier restricting to frequencies of $F_j$. 
Using Lemma \ref{lem:1}, we know that 
$$ \left\| T_{\chi_{\Omega^j_1} m_1/\phi}  \right\|_{L^{2} \times L^{\infty} \to L^\infty} \lesssim 2^{-j} 2^{j(\frac{d+1}{4})}.$$
Consequently, with $|F_j|\lesssim 2^{j}$ we deduce that 
\begin{align*}
 \|I_t(f,g)\|_{L^\infty}  & \lesssim \sum_{j \leq 0} 2^{j(\frac{d+1}{4})-j+j\frac{1}{2}} t^{-\frac{d}{2}} \|f\|_{L^{1}} \| g \|_{L^{1}} \\
   & \lesssim  t^{-\frac{d}{2}} \|f\|_{L^{2,s}} \| g \|_{L^{2,s}},
\end{align*}
as soon as $s>\frac{d}{2}$ (since $d\geq 2$). \\
Let us now deal with the second term $II_t$: 
\begin{align*}
 \|II_t(f,g)\|_{L^\infty} & \lesssim t^{-d/2} \|T_{m_1/\phi}(f,g) \|_{L^1} \\
 & \lesssim  t^{-d/2} \sum_{j \leq 0}   \|T_{m_1\chi_{\Omega^j_1}/\phi}\|_{L^{2} \times L^{1} \to L^{1}}   \|P_{F_j} f \|_{L^{2}} \| g \|_{L^{1}}.
\end{align*}
We conclude as previously using $|F_j|\lesssim 2^{j}$ and Lemma \ref{lem:1}. So for $s>\frac{d}{2}$, we obtain
\begin{align*}
 \|II_t(f,g)\|_{L^\infty}  & \lesssim  t^{-d/2}  \|f\|_{L^{2,s}} \|g\|_{L^{2,s}}.
 \end{align*}
\end{proof}

\begin{lemma} \label{lem:1} Consider $\chi_{\Omega^j_1}$ as above. It satisfies 
$$ \left\| T_{\chi_{\Omega^j_1}}  \right\|_{L^{2} \times L^{\infty} \to L^\infty} + \left\| T_{\chi_{\Omega^j_1}}  \right\|_{L^{2} \times L^{1} \to L^1} \lesssim  2^{j(\frac{d+1}{4})}.$$ 
\end{lemma}

\begin{proof}
The dual of the first estimate is the second one, with the symbol $\chi_{\Omega^j_1}(\xi,\eta)$ replaced by $\chi_{\Omega^j_1}(\eta-\xi,\eta)$, and one checks that the same proof as below applies to the symbol $\chi_{\Omega^j_1}(\eta-\xi,\eta)$. 

For this reason, we only prove the first estimate, following ideas in \cite[Proposition 5.3]{BG2}. Let us simplify the notation and write $T_j$ for the first adjoint of $T_{\chi_{\Omega^j_1}}$ such that 
$$ \left\| T_{\chi_{\Omega^j_1}}  \right\|_{L^{2} \times L^{\infty} \to L^\infty} = \left\| T_j  \right\|_{L^{1} \times L^{\infty} \to L^2}.$$
So $T_j$ is the bilinear Fourier multiplier associated to the symbol $(\xi,\eta) \to \chi_{\Omega^j_1}(-\eta,-\xi)$.
In order to estimate the $L^2$-norm of $T_j$, we use a $TT^*$ argument:
$$
\left\| T_j (f,g) \right\|_{L^2}^2 =\iiiint f(x^1) \overline{f}(x^2) g(y^1) \overline{g}(y^2) K_j(x^1-y^1,y^2-x^2,y^1-y^2)\,dx^1\,dx^2\,dy^1\,dy^2$$
where
$$
K_j(a,b,c) = \frac{1}{(2\pi)^d} \iiint \chi_{\Omega^j_1}(\eta,\xi) \chi_{\Omega^j_1}(\zeta,\xi) e^{-i(\eta a + \zeta b + \xi c)} d\eta \, d\zeta \, d\xi.
$$
Thus
\begin{equation*}
\begin{split}
& \left\| T_j (f,g) \right\|_{L^2}^2 \\
& \qquad \leq \|g\|_{L^\infty}^2 \iint |f(x^1)| |f(x^2)| \iint \left| K_j(x^1-y^1,y^2-x^2,y^1-y^2) \right| \,dy^1\,dy^2\,dx^1\,dx^2, \\
& \qquad \lesssim  \left(\sup_{x^1,x^2} \iint \left| K_j(x^1-y^1,y^2-x^2,y^1-y^2)\right| \,dy^1\,dy^2\right) \|g\|_{L^\infty}^2 \|f\|_{L^1}^2.
\end{split}
\end{equation*}
Everything now boils down to estimating
$$
\iint | K_j(x^1-y^1,y^2-x^2,y^1-y^2) | \,dy^1\,dy^2.
$$
A change of variables gives
\begin{align*}
K_j(x^1-y^1,y^2-x^2,y^1-y^2) & = \frac{1}{(2\pi)^{d}} \iint F_{x}(\alpha,\beta) e^{i\alpha(x^1-y^1)} e^{i\beta(y^2-x^2)} \,d\alpha\, d\beta \\
 & =\frac{1}{(2\pi)^{d/2}} \widehat{F_x} (y^1-x^1,x^2-y^2)
\end{align*}
where
\begin{equation}
\label{lapin}
F_{x}(\alpha,\beta) \overset{def}{=} \int \chi_{\Omega^j_1} (\xi-\alpha,\xi) \chi_{\Omega^j_1} (\xi-\beta,\xi) e^{ix\xi} \,d\xi\;\;\;\;\mbox{and}\;\;\;\;x \overset{def}{=} x^2 - x^1.
\end{equation}
Combining the last few lines,
\begin{align}
\left\| T_{\chi_{\Omega^j_1}}  \right\|_{L^{2} \times L^{\infty} \to L^\infty} & = \left\| T_j  \right\|_{L^{1} \times L^{\infty} \to L^2} \nonumber \\ 
 & \lesssim \left(\sup_x \|\widehat{F_x}\|_{L^1}\right)^{\frac{1}{2}}. \label{lievre}
\end{align}
Define 
$$ O^j\overset{def}{=} \{(\alpha,\xi), (\xi-\alpha,\xi)\in \Omega_1^j\}$$.
\begin{itemize}
\item First observe that $O^j$ is included in a $\sim 2^j$ neighborhood of the sub-manifold $\{(\xi,\eta), \nabla_2 \phi(\xi-\alpha,\xi)=0\}$ (here, we denoted $\nabla_2 \phi$ for the derivative of $\phi(\cdot,\cdot)$ with respect to the second variable).
Since 
$$\nabla_\xi \nabla_2 \phi(\xi-\alpha,\xi)=\operatorname{Hess}[b] (\xi) \quad \mbox{and} \quad \nabla_\alpha \nabla_2 \phi(\xi-\alpha,\xi)= \operatorname{Hess}[c](-\alpha)$$ 
and due to the non-degeneracy of the Hessian matrices of $b$ and $c$, this sub-manifold can be parameterized by $\xi$ or $\alpha$. Thus, for $\alpha$ fixed, $\chi_{\Omega^j_1}(\xi-\alpha,\xi)$ is not zero for $\xi$ in a ball of radius $\sim 2^j$. This implies that
\begin{equation}
\label{gelinotte}
\|F_x\|_{L^\infty(\R^{2d})} \lesssim 2^{dj}.
\end{equation}
\item Next, observe that the projection on the $\alpha$ variable of $O^j$ has size $\sim 2^j$. To see this, notice that $O^j$ is contained in a $\sim 2^j$ neighborhood of the set defined by 
\begin{equation}
\label{tourterelle}
\phi(\xi-\alpha,\xi)=0 \quad \mbox{and} \quad \nabla_2 \phi(\xi-\alpha,\xi) = 0.
\end{equation}
The linearization of these two conditions around a point $(\alpha_0,\xi_0)$ is as follows: the point $(\alpha_0,\xi_0)+(d\alpha,d\xi)$ satisfies them if
$$
\nabla_1 \phi(\xi_0-\alpha_0,\xi_0) \cdot(d\xi - d\alpha) = 0 \quad \mbox{and} \quad \operatorname{Hess}[b] (\xi_0) d\xi +  \operatorname{Hess}[c](-\alpha_0) d\alpha = 0,
$$
which is equivalent to
$$
\left\{ \begin{array}{l}
(\operatorname{Hess}[b](\xi_0))^{-1}(\operatorname{Hess}[b] (\xi_0) + \operatorname{Hess}[c](-\alpha_0))d\alpha \in (\nabla_1 \phi(\xi_0-\alpha_0,\xi_0))^\perp \\
\operatorname{Hess}[b] (\xi_0) d\xi + \operatorname{Hess}[c](-\alpha_0) d\alpha = 0.
\end{array} \right.
$$
(notice that $\operatorname{Hess}[b ](\xi_0) + \operatorname{Hess}[c](-\alpha_0) = \operatorname{Hess}_\eta [\phi](\xi_0-\alpha_0,\xi_0)$, which, by (A1), is invertible).
These equations show that the projection of the linearization of the set defined by (\ref{tourterelle}) in the $\alpha$ variable is $(d-1)$ dimensional. By the implicit function theorem, this remains true for the projection of the set (\ref{tourterelle}). Since $O^j$ is a $\sim 2^j$ neighborhood of this set, we obtain that its projection in the $\alpha$ variable has size $\sim 2^j$. Finally, if $(\alpha,\xi)$ and $(\beta,\xi)$ belong to $O^j$, we already saw that $\alpha = \beta + O(2^j)$. Therefore
\begin{equation}
\label{gelinotte2}
|\operatorname{Supp} F_x | \lesssim 2^{j(1+d)}.
\end{equation}
\end{itemize}
The estimates~(\ref{gelinotte}) and~(\ref{gelinotte2}) imply that
$$
\|F_x\|_{L^2(\R^{2d})} \lesssim \|F_x\|_{L^\infty(\R^{2d})} |\operatorname{Supp} F_x |^{1/2} \lesssim 2^{j\frac{3d+1}{2}}.
$$
Using in addition the bound $|\partial_\xi^{\alpha} \partial_\eta^\beta m|\lesssim 2^{-j(|\alpha|+|\beta|)}$ gives
$$
\| |\cdot|^{2d} \widehat{F_x} \|_{L^2(\R^{2d})} = \left\| \partial^{2d}_{\alpha,\beta} F_{x} \right\|_{L^2(\R^{2d})} \lesssim 2^{j\frac{1-d}{2}}
$$
The two last estimates give
$$
\|\widehat{F_x}\|_{L^1(\R^{2d})} \lesssim \| \widehat{F_x} \|_{L^2(\R^{2d})}^{1/2} \| |\cdot|^{2d} \widehat{F_x} \|_{L^2(\R^{2d})}^{1/2} \lesssim 2^{j\frac{d+1}{2}},
$$
from which the desired bound follows.
\end{proof}

\subsection{Study of $u_2$ corresponding to time-resonances}

We now focus on the piece coming from time-resonances:
$$ u_2(t) = \int_0^t  \int_{\R^{2d}}  e^{ix\xi} e^{ita(\xi)} e^{is\phi(\xi,\eta)} m_2(\xi,\eta) \widehat{f}(\eta) \widehat{g}(\xi-\eta) \,d\xi \,d\eta\,ds,$$
which can be written by setting $s=\sigma t$ and $X=\frac{x}{t}$ as
\begin{align*}
 u_2(t) & = t \int_0^1  \int_{\R^{2d}}  e^{it(X\xi +a(\xi)+ \sigma \phi(\xi,\eta))} m_2(\xi,\eta) \widehat{f}(\eta) \widehat{g}(\xi-\eta) \,d\xi \,d\eta\,d\sigma  \\
 & = t \int K(y,z) f(y) g(z) \, dy\, dz,
\end{align*}
where the bilinear kernel $K$ is given by
\begin{equation}
\label{kiwi}
K(y,z)\overset{def}{=} \frac{1}{(2\pi)^d} \int_0^1  \int_{\R^{2d}}  e^{it(X\xi +a(\xi)+ \sigma \phi(\xi,\eta))} m_2(\xi,\eta) e^{-iy\eta} e^{-iz(\xi-\eta)} \,d\xi \,d\eta\,d\sigma.\end{equation}

\begin{theorem} \label{thm:K2} For $d=2$, there exists a decomposition of the kernel $K=\widetilde{K}+L$ such that
\begin{itemize}
\item the kernel $\widetilde{K}(y,z)$ is (uniformly with respect to $(y,z)$) bounded by
$$ |\widetilde{K}(y,z)| \lesssim \frac{\log(t)}{t^2}[1+|y|+|z|]^\epsilon $$
for every $\epsilon>0$;
\item the bilinear operator ${\mathcal E}$ associated to the kernel $L$ satisfies
$$ \|{\mathcal E}\|_{L^{2,s} \times L^{2,s} \to L^\infty} \lesssim t^{-2},$$
for every $s>\frac{3}{2}$.
\end{itemize}
Consequently,
$$ \|u_2(t)\|_{L^\infty} \lesssim \frac{\log(t)}{t}  \|f\|_{L^{2,s}} \|g\|_{L^{2,s}} $$  %+ \|f\|_{L^{2,1+\epsilon}}\|g\|_{L^{2,2+\epsilon}} \right] $$
for every $s>\frac{3}{2}$.
\end{theorem}

\begin{proof} \underline{Step 1: decomposition of $K$.} We first distinguish between times close to zero, and times away from zero. For a constant $\epsilon_0>0$ which will be taken sufficiently small, we can write, starting from the definition~(\ref{kiwi})
\begin{equation*}
\begin{split}
K(y,z) & = \frac{1}{(2\pi)^d} \int_0^{\epsilon_0} \int_{\R^{2d}}  \dots \,d\xi \,d\eta\,d\sigma + \frac{1}{(2\pi)^d} \int_{\epsilon_0}^1 \int_{\R^{2d}}\dots \,d\xi \,d\eta\,d\sigma \\
& \overset{def}{=} K_{[0,\epsilon_0]}(y,z) + K_{[\epsilon_0,1]}(y,z).
\end{split}
\end{equation*}
Next, recall that the symbol $m_2$ is supported in $\Omega_2$, which is essentially a cone whose top is a sub-manifold of dimension $d-1$. As in the previous section, we split $\Omega_2$ into almost disjoint sub-sets $\Omega^j_2$ defined by $\Omega^j_2 \overset{def}{=} \Omega_2\cap \{ 2^{j-1} \leq |\partial_\eta \phi| <2^{j+2} \}$, and define an associated smooth partition of unity $\chi_{\Omega^j_2}$ such that
$$
\sum_{j \leq 0} \chi_{\Omega^j_2}=1 \;\;\mbox{on} \;\; \Omega_2, \quad \operatorname{Supp} \chi_{\Omega_2^j} \subset \Omega_2^{j}, \quad \mbox{and} \quad |\partial_\xi^\alpha \partial_\eta^\beta \chi_{\Omega_2^j}(\xi,\eta)| \lesssim 2^{-j(|\alpha|+|\beta|)}.
$$
That procedure gives rise to a decomposition of the kernel
$$ K_{[0,\epsilon_0]} = \sum_{j\leq 0} K^j,$$
with
\begin{align*}
 K^j(y,z) & \overset{def}{=} \frac{1}{(2\pi)^d} \int_0^{\epsilon_0}  \int_{\R^{2d}}  e^{it(X\xi +a(\xi)+ \sigma \phi(\xi,\eta))} (m_2 \chi_{\Omega^j_2})(\xi,\eta) e^{-iy\eta} e^{-iz(\xi-\eta)} \,d\xi \,d\eta\,d\sigma \\
 &  = \frac{1}{(2\pi)^d} \int_0^{\epsilon_0}  \int_{\R^{2d}}  e^{it\psi(\xi,\eta,\sigma)} (m_2\chi_{\Omega^j_2})(\xi,\eta) e^{-i(y-z)\eta-iz\xi} \,d\xi \,d\eta\,d\sigma,
\end{align*}
where we recall that $\psi(\xi,\eta,\sigma)= X\xi +a(\xi)+ \sigma \phi(\xi,\eta)$. This phase function satisfies $\nabla_\xi \psi(\xi,\eta,\sigma,X) = X+\nabla_\xi a(\xi)+\sigma\nabla_\xi \phi(\xi,\eta)$ and $\textrm{Hess}_\xi [\psi](\xi,\eta,\sigma,X) = \textrm{Hess}[a](\xi)+\sigma \textrm{Hess}_\xi [\phi(\xi,\eta)]$. Since $a$ is assumed to have a non-degenerate Hessian matrix everywhere, we can take $\epsilon_0$ sufficiently small so that, for $\sigma \in [0,\epsilon_0]$, the Hessian $\textrm{Hess}_\xi [\psi](\xi,\eta,\sigma)$ is non-degenerate. As a consequence, for $\eta,\sigma$ fixed the equation $\nabla_\xi \psi(\xi,\eta,\sigma)=0$ has at most one solution $\xi\overset{def}{=}\Xi(\eta,\sigma,X)$. Introduce a small parameter $r = r(j) \geq 2^j$ (which will be fixed later) and split the kernel $K^j$ as
$$ K^j = K^j_1+K^j_2$$
with
\begin{align*}
& K^j_1(y,z)\overset{def}{=} \\
& \quad \frac{1}{(2\pi)^d} \int_0^{\epsilon_0}  \int_{\R^{2d}}  e^{it\psi(\xi,\eta,\sigma)} (m_2\chi_{\Omega^j_2})(\xi,\eta) e^{-i(y-z)\eta-iz\xi}  \chi\left(\frac{\nabla_\xi \psi - t^{-1}z}{r}\right)\,d\xi \,d\eta\,d\sigma
\end{align*}
and
\begin{equation}\label{eq:K2} 
\begin{split}
& K^j_2(y,z)\overset{def}{=} \\
&\quad \frac{1}{(2\pi)^d}  \int_0^{\epsilon_0}  \int_{\R^{2d}}  e^{it\psi(\xi,\eta,\sigma)} (m_2\chi_{\Omega^j_2})(\xi,\eta) e^{-i(y-z)\eta-iz\xi}  \left[1-\chi\left(\frac{\nabla_\xi \psi -t^{-1}z}{r}\right)\right]\,d\xi \,d\eta\,d\sigma.
\end{split}
\end{equation}
Here $\chi$ is a smooth function, supported on $[-2,2]^d$ and equal to $1$ on $[-1,1]^d$. Moreover, the quantity $\nabla_\xi \psi$ is equivalent to $(\operatorname{Hess}_\xi [\psi](\Xi)) (\xi-\Xi)$ so the smooth cutoff in $K^j_1$ may be thought of as a localization of the frequency $\xi$ to a ball of radius $\sim r$. For another small parameter $\rho = \rho(t,j) <\epsilon_0$, decompose the kernel $K^j_1$ into $K^j_{1,1}+K^j_{1,2}$ with
\begin{align*}
 \lefteqn{K^j_{1,1}(y,z)\overset{def}{=}} & & \\ 
 &  &\frac{1}{(2\pi)^d} \int_0^{\rho}  \int_{\R^{2d}}  e^{it\psi(\xi,\eta,\sigma)} (m_2\chi_{\Omega^j_2})(\xi,\eta) e^{-i(y-z)\eta-iz\xi}  \chi\left(\frac{\nabla_\xi \psi -t^{-1}z}{r}\right) \,d\xi \,d\eta\,d\sigma
\end{align*}
and
\begin{align*}
 \lefteqn{K^j_{1,2}(y,z)\overset{def}{=}} & & \\
 & &\frac{1}{(2\pi)^d}  \int_\rho^{\epsilon_0}  \int_{\R^{2d}}  e^{it\psi(\xi,\eta,\sigma)} (m_2\chi_{\Omega^j_2})(\xi,\eta) e^{-i(y-z)\eta-iz\xi}  \chi\left(\frac{\nabla_\xi \psi -t^{-1}z}{r}\right) \,d\xi \,d\eta\,d\sigma.
\end{align*}
To estimate the kernel $K^j_{1,2}$, we use an integration by parts in $\eta$, using $i\sigma t \nabla_\eta \phi e^{it\psi} = \nabla_\eta e^{it\sigma \psi}$ (due to the definition of $\psi$). This gives:
\begin{align*}
K^j_{1,2}\overset{def}{=} \overline{K}^j_{1,2} + i(z_p-y_p) R^j_{1,2}
\end{align*}
with
\begin{align*} 
 \lefteqn{\overline{K}^j_{1,2}(y,z) \overset{def}{=}} & & \\
 & & \frac{1}{(2\pi)^d}  \int_\rho^{\epsilon_0}  \int_{\R^{2d}}  e^{it\psi(\xi,\eta,\sigma)} \partial_{\eta_p}\left[\frac{m_2 \chi_{\Omega^j_2}}{it\sigma \partial_{\eta_p} \phi} \, \chi\left(\frac{\nabla_\xi \psi -t^{-1}z}{r}\right) \right]  e^{-i(y-z)\eta-iz\xi} \,d\xi \,d\eta\,d\sigma, 
\end{align*}
and
\begin{align*} 
 \lefteqn{R^j_{1,2}(y,z) \overset{def}{=}} & & \\
 & &\frac{1}{(2\pi)^d}  \int_\rho^{\epsilon_0}  \int_{\R^{2d}}  e^{it\psi(\xi,\eta,\sigma)} \left[\frac{m_2 \chi_{\Omega^j_2}}{it\sigma \partial_{\eta_p} \phi} \, \chi\left(\frac{\nabla_\xi \psi -t^{-1}z}{r}\right)  \right]  e^{-i(y-z)\eta-iz\xi} \,d\xi \,d\eta\,d\sigma, 
\end{align*}
where we choose the coordinate $\eta_p$ such that $|\nabla_\eta \phi| \simeq |\partial_{\eta_p} \phi| \simeq 2^j$ (which is possible up to some other truncations).

\bigskip 
We finally set $r(j) = 2^j$, $\rho(t,j) = (t2^{2j})^{-1}$, 
$$ 
\widetilde{K} \overset{def}{=} \sum_{2^j\leq (\epsilon_0 t)^{-1/2}} K^j + \sum_{2^j\geq (\epsilon_0 t)^{-1/2}} \left[ K^j_{1,1} + \overline{K}^j_{1,2} + K^j_2 \right] + K_{[\epsilon_0,1]},
$$
and
$$
{\mathcal E} \overset{def}{=} \sum_{2^j\geq (\epsilon_0 t)^{-1/2}} {\mathcal E}_j \quad \mbox{with} \quad
{\mathcal E}_j(f,g)\overset{def}{=} i \left(T_{R^j_{1,2}}(x_pf,g) + T_{R^j_{1,2}}(f,x_p g)\right).
$$
Here $p\in\{1,...,d\}$ is a coordinate chosen as indicated above, $x_pf$ stands for the function $x\to x_p f(x)$ and $T_{R^j_{1,2}}$ is the bilinear operator associated to the bilinear kernel.

\bigskip

\noindent \underline{Step 2: estimating $\widetilde{K}$ and $\mathcal{E}$.}
We postpone the estimates of the various pieces in the decomposition of $K_{[0,\epsilon_0]}$ to the next subsection (Lemmas \ref{lem:u_3}, \ref{lem:K_11}, \ref{lem:K_12}, \ref{lem:E} and \ref{lem:K_2}), and show here how to conclude from them. Recall that $d=2$, $r=2^j$, $\rho=(t2^{2j})^{-1}$; we choose further $N=2$.

First, the desired estimate on $\mathcal{E}$ follows from Lemma~\eqref{lem:E}, which gives
$$
\|\mathcal{E}\|_{L^{2,s} \times L^{2,s} \rightarrow L^\infty} \leq \sum_{(\epsilon_0 t)^{-1} \leq 2^{2j} \leq 1} \|\mathcal{E}_j\|_{L^{2,s} \times L^{2,s} \rightarrow L^\infty} \lesssim \sum_{2^j \leq 1} t^{-1-\frac{d}{2}} 2^{j\nu} \lesssim t^{-1-\frac{d}{2}}.
$$ 
Next, Lemma \ref{lem:u_3} gives 
\begin{equation}
\label{martinpecheur1}
\begin{split}
\sum_{2^{2j} \leq (\epsilon_0 t)^{-1}} |K^j_1(y,z)|+|K^j_2(y,z)| & \lesssim \sum_{2^j \leq t^{-1}} t^{-1} 2^{jd} +  \sum_{t^{-1} \leq 2^j \leq (\epsilon_0 t)^{-\frac{1}{2}}} t^{-1} 2^{dj} \langle \log(t2^{j}) \rangle \\
 & \lesssim t^{-(d+1)} + t^{-1} t^{-\frac{d}{2}} \log(t) \lesssim t^{-2} \log(t).
\end{split}
\end{equation}
If $2^{2j}\geq (\epsilon_0 t)^{-1}$, then $\rho=(t2^{2j})^{-1}\leq \epsilon_0$ and  Lemmas \ref{lem:K_11} and \ref{lem:K_12} yield for every $\epsilon>0$
\begin{align*}
|K^j_{1,1}(y,z)| + |\overline{K}^j_{1,2}(y,z)| & \lesssim \frac{\rho r 2^j}{t} + \frac{1}{t^{2+\epsilon}} \frac{1}{(\rho 2^j)^\epsilon} [2^{-j} + |y|+|z|]^{\epsilon} \\
 & \lesssim \frac{1}{t^2} [1 + 2^j(|y|+|z|)]^{\epsilon}.
\end{align*}
Summing over $j$ such $(\epsilon_0 t)^{-1} \leq 2^{2j} \lesssim 1$ gives an extra $\log(t)$ quantity and hence
\begin{equation}
\label{martinpecheur2}
\sum_{(\epsilon_0 t)^{-1} \leq 2^{2j} \leq 1}|K^j_{1,1}(y,z)| + |\overline{K}^j_{1,2}(y,z)| \lesssim \frac{1}{t^2}\left[ \log(t) + (|y|+|z|)^{\epsilon}\right].
\end{equation}
For $K^j_2$, using Lemma \ref{lem:K_2} (with $N=2$) gives 
\begin{align}
\label{martinpecheur3}
\sum_{(\epsilon_0 t)^{-1} \leq 2^{2j} \leq 1} |K^j_2(y,z)| \lesssim  \sum_{(\epsilon_0 t)^{-1} \leq 2^{2j} \leq 1} \frac{1}{t^2} \lesssim \log(t) t^{-2}.
\end{align}
Combining~\eqref{martinpecheur1}, \eqref{martinpecheur2} and \eqref{martinpecheur3} gives the desired estimate on $\widetilde{K}$.

\bigskip

\noindent \underline{Step 3: estimating $K_{[\epsilon_0,1]}$.}
It remains to estimate $K_{[\epsilon_0,1]}$. By the assumption (A3), the interval $[\epsilon_0,1]$ can be split into $[\epsilon_0,1] = I \cup J$, where I and J are finite unions of intervals such that: on $I$, $\operatorname{Hess}_{(\xi,\eta)} [a(\xi)+\sigma \phi(\xi,\eta)]$ is not degenerate, whereas on $J$, $\operatorname{Hess}_\xi [ a(\xi)+\sigma \phi(\xi,\eta)]$ is not degenerate. Split accordingly
\begin{equation*}
\begin{split}
K_{[\epsilon_0,1]}(y,z) & = \frac{1}{(2\pi)^d} \int_I \int_{\R^{2d}}  \dots \,d\xi \,d\eta\,d\sigma + \frac{1}{(2\pi)^d} \int_J \int_{\R^{2d}}\dots \,d\xi \,d\eta\,d\sigma \\
& \overset{def}{=} K_I(y,z) + K_J(y,z).
\end{split}
\end{equation*}
Consider first $K_I(y,z)$. It is given by
$$
K_I(y,z)\overset{def}{=} \frac{1}{(2\pi)^d} \int_I \int_{\R^{2d}}  e^{it(X\xi +a(\xi)+ \sigma \phi(\xi,\eta))} m_2(\xi,\eta) e^{-iy\eta} e^{-iz(\xi-\eta)} \,d\xi \,d\eta\,d\sigma.
$$
Since $\operatorname{Hess}_{(\xi,\eta)} [a(\xi)+\sigma \phi(\xi,\eta)]$ is not degenerate, the stationary phase lemma in the $(\xi,\eta)$ variables gives
$$
|K_I(y,z)| \lesssim \int_I \frac{1}{t^{2}} \,d\sigma \lesssim \frac{1}{t^2}.
$$
Finally, there remains $K_J$. It can be dealt with exactly as $K_{[0,\epsilon_0]}$, but in a simpler way, since the singularity at $\sigma=0$ is absent. We therefore skip this term.
\end{proof}

\medskip

Let us now give a $3$-dimensional version of Theorem \ref{thm:K2}:
\begin{theorem} \label{thm:K3} For $d=3$, there exists a decomposition of the kernel $K=\widetilde{K}+L$ such that
\begin{itemize}
\item the kernel $\widetilde{K}(y,z)$ is (uniformly with respect to $(y,z)$) bounded by
$$ |\widetilde{K}(y,z)| \lesssim t^{-\frac{5}{2}} [1+|y|+|z|]^\epsilon $$
for every $\epsilon>0$;
\item the bilinear operator ${\mathcal E}$ associated to the kernel $L$ satisfies
$$ \|{\mathcal E}\|_{L^{2,s} \times L^{2,s} \to L^\infty} \lesssim t^{-\frac{5}{2}},$$
for every $s>\frac{3}{2}$.
\end{itemize}
Consequently,
$$ \|u_2(t)\|_{L^\infty} \lesssim t^{-\frac{3}{2}}  \|f\|_{L^{2,s}} \|g\|_{L^{2,s}} $$ 
for every $s>\frac{3}{2}$.
\end{theorem}

\begin{proof} We perform exactly the same decomposition as in the step 1 of the proof of Theorem~\ref{thm:K2}. We set then $r(j)=2^{\alpha j}$ (for some $\alpha\in(\frac{1}{2},1)$ sufficiently close to $1$), $N=\frac{5}{2}$,
$$
\widetilde{\rho}(t,j)=2^{j\mu} (t2^{j(1+2\alpha)})^{-1} \quad \mbox{and} \quad \rho(t,j) = \min(\widetilde{\rho},\epsilon_0)
$$ 
(for a sufficiently small parameter $\mu>0$).

Define then
$$ 
\widetilde{K} \overset{def}{=} \sum_{j\leq 0,\ \widetilde{\rho} \geq \epsilon_0} K^j + \sum_{j \leq 0,\ \widetilde{\rho} \leq \epsilon_0} \left[ K^j_{1,1} + \overline{K}^j_{1,2} + K^j_2 \right] + K_{[\epsilon_0,1]},
$$
and
$$
{\mathcal E} \overset{def}{=} \sum_{j \leq 0,\ \widetilde{\rho} \leq \epsilon_0} {\mathcal E}_j \quad \mbox{with} \quad
{\mathcal E}_j(f,g)\overset{def}{=} i \left(T_{R^j_{1,2}}(x_pf,g) + T_{R^j_{1,2}}(f,x_p g)\right).
$$

The operator $K_{[\epsilon_0,1]}$ can be treated exactly as in Theorem~\ref{thm:K2}, so we skip this argument. As for $\mathcal{E}$, it is controlled by Lemma \ref{lem:E} provided $\alpha$ is chosen sufficiently close to $1$:
$$
\| \mathcal{E} \|_{L^{2,s} \times L^{2,s} \rightarrow L^\infty} \lesssim \sum_{j \leq 0,\ \widetilde{\rho} \leq \epsilon_0} \| \mathcal{E}_j \|_{L^{2,s} \times L^{2,s} \rightarrow L^\infty} \lesssim \sum_{j \leq 0} t^{-5/2} 2^{\frac{2}{3}(\alpha-1) j} 2^{j \nu} \lesssim t^{-5/2}.
$$
Thus, it only remains to deal with $\widetilde{K}$, to which we now turn. First, using that for $\rho=\epsilon_0$, $K^j_1 = K^j_{1,1}$, Lemma \ref{lem:K_11} gives
\begin{align*}
\sum_{j\leq 0,\ \widetilde{\rho} \geq \epsilon_0}  |K^j_{1}(y,z)| \lesssim \sum_{j\leq 0,\ \widetilde{\rho} \geq \epsilon_0} \frac{2^{j(1+2\alpha)}}{t^{3/2}} \lesssim t^{-\frac{5}{2}} \sum_{j\leq 0,\ \widetilde{\rho} \geq \epsilon_0} 2^{j\mu} \lesssim t^{-\frac{5}{2}}.
\end{align*}
Next, still using Lemma \ref{lem:K_11}, we obtain 
\begin{align*}
\sum_{j\leq 0,\ \widetilde \rho\leq \epsilon_0} |K^j_{1,1}(y,z)| & \lesssim \sum_{j\leq 0,\ \widetilde \rho\leq \epsilon_0} \frac{\rho  2^{j(1+2\alpha)}}{t^{\frac{3}{2}}} \lesssim t^{-\frac{5}{2}} \sum_{j\leq 0} 2^{j \mu} \lesssim t^{-\frac{5}{2}},
\end{align*}
whereas Lemma \ref{lem:K_12} yields
\begin{align*}
\sum_{j\leq 0,\ \widetilde \rho\leq \epsilon_0} & |\overline{K}^j_{1,2}(y,z)| \lesssim \sum_{j\leq 0,\ \widetilde \rho\leq \epsilon_0} \frac{2^{j(2\alpha-1-\epsilon \mu +2\epsilon\alpha)}}{t^{\frac{5}{2}}} [2^{-j} + |y|+|z|]^{\epsilon} \\
& \lesssim t^{-\frac{5}{2}} \sum_{j\leq 0,\ \widetilde \rho\leq \epsilon_0} 2^{j(2\alpha-1-\epsilon \mu +2\epsilon\alpha-\epsilon)} + t^{-\frac{5}{2}}[|y|+|z|]^{\epsilon} \sum_{j\leq 0,\ \widetilde \rho\leq \epsilon_0} 2^{j(2\alpha-1-\epsilon \mu +2\epsilon\alpha)}  \\
 & \lesssim t^{-\frac{5}{2}}[1+(|y|+|z|)^{\epsilon}],
 \end{align*}
by choosing $\mu$ sufficiently small. For $K^j_2$, we use Lemma \ref{lem:K_2} (with $N=\frac{5}{2}$) to get
\begin{align*}
\sum_{j\leq 0} |K^j_2(y,z)| \lesssim  & \sum_{j\leq 0} (t2^{j(1+\alpha)})^{-\frac{5}{2}} 2^{j(4+\alpha)} \lesssim t^{-\frac{5}{2}},
\end{align*}
since $\alpha<1$.

\mb Putting all these estimates together gives $|\widetilde{K}(y,z)| \lesssim t^{-\frac{5}{2}} \left[ 1 + (|y|+|z|)^{1+\epsilon}\right]$ uniformly in $(y,z)$ for every $\epsilon>0$, which is the desired estimate.
\end{proof}

\subsection{Intermediate estimates}

First, we start by recalling the following useful stationary phase result : 

\begin{lemma} \label{lem:oscillatory}
Let $K$ be a compact subset of $\mathbb{R}^n$. Then if $F$ is compactly supported in $K$, $\lambda$ is smooth and has a unique stationary point at $0$ with a non-degenerate Hessian, then uniformly in $\theta \in \mathbb{R}^n$,
$$
\left|\int_{\R^n} e^{it \lambda(x)} e^{ix\theta} F(x) dx\right| \lesssim t^{-\frac{n}{2}} \|\widehat{F}\|_{L^1}.
$$ 
\end{lemma}

\begin{lemma} \label{lem:u_3} For $2^j\leq t^{-\frac{1}{2}}$, the kernels $K^j_1$ and $K^j_2$ satisfy the following easy estimates:  
$$ \left|K^j_1(y,z)\right|+\left|K^j_2(y,z)\right| \lesssim \frac{2^{dj}}{t} {\bf 1}_{2^{j}\leq t^{-1}} + \frac{2^{dj}}{t} \langle \log(t2^{j}) \rangle {\bf 1}_{t^{-1} \leq 2^{j}}.$$
\end{lemma}

\begin{proof}
We only deal with $K^j_1$, since the exact same proof applies to $K^j_2$. We recall that 
\begin{align*}
& K^j_1(y,z) \\
& \ \ = \frac{1}{(2\pi)^d} \int_0^{\epsilon_0}  \int_{\R^{2d}}  e^{i t\left(X\xi +a(\xi)+ \sigma \phi(\xi,\eta)\right)} (m_2 \chi_{\Omega^j_2}) e^{-i(y-z)\eta-iz\xi} \chi\left(\frac{\nabla_\xi \psi -t^{-1}z}{r}\right)\,d\xi \,d\eta\,d\sigma \\ 
& \ \ =  \frac{1}{(2\pi)^d}  \int_{\R^{2d}}  e^{it\left(X\xi +a(\xi)\right)} \frac{e^{i\epsilon_0 t\phi(\xi,\eta)}-1}{it\phi(\xi,\eta)} (m_2 \chi_{\Omega^j_2}) e^{-i(y-z)\eta-iz\xi}  \chi\left(\frac{\nabla_\xi \psi -t^{-1}z}{r}\right)\,d\xi \,d\eta,
\end{align*}
so that
\begin{align*} \left|K^j_1(y,z)\right| & \lesssim \int_{\R^{2d}}  \left|\frac{e^{i\epsilon_0 t\phi(\xi,\eta)}-1}{t\phi(\xi,\eta)}\right| \chi_{\Omega^j_2}(\xi,\eta) \,d\xi \,d\eta \\
 & \lesssim \int_{\Omega^j_2} \max(1, t|\phi(\xi,\eta)|)^{-1} \, d\xi d\eta.
 \end{align*}
Let $\Omega_{k,j}=\{ (\xi,\eta)\in \Omega^j_2,\ 2^k \leq |\phi|< 2^{k+1}\}$. Then $|\Omega_{k,j}| \lesssim 2^{k+dj}$. Hence,
\begin{align*}  \int_{\Omega^j_2} \max(1, t|\phi(\xi,\eta)|)^{-1} \, d\xi\, d\eta & \lesssim \sum_{2^k \leq t^{-1},\ 2^k \lesssim 2^j}   2^{k+dj} + \sum_{t^{-1} \leq 2^k \lesssim 2^j } t^{-1}2^{dj} \end{align*}
which gives the desired bound.
\end{proof}

\begin{lemma} \label{lem:K_11} For $d\in\{2,3\}$, $\rho \leq \epsilon_0$ and $r\geq 2^j$, the kernel $K^j_{1,1}$ satisfies the following estimate: 
$$ \left|K^j_{1,1}(y,z)\right| \lesssim \frac{\rho r^{d-1} 2^j}{t^{\frac{d}{2}}}.$$
\end{lemma}

\begin{proof}
We set
$$ A_j(\eta,\sigma,y,z) \overset{def}{=}  \int_{\R^{d}} a_j(\xi,\eta,\sigma,y,z)\,d\xi $$
with
$$
a_j(\xi,\eta,\sigma,y,z) \overset{def}{=} e^{it\psi(\xi,\eta,\sigma)} m_2(\xi,\eta) \chi_{\Omega^j_2}(\xi,\eta) e^{-i(y-z)\eta - iz\xi}  \chi\left(\frac{\nabla_\xi \psi -t^{-1}z}{r}\right)
$$
so that $\displaystyle K^j_{1,1} = \int_{0}^\rho \int_{\mathbb{R}^d} A_j d\eta\, d\sigma.$
For every fixed $(\eta,\sigma)$, we have
$$
\left| \operatorname{Supp} \chi_{\Omega^j_2}(\cdot,\eta) \right| \lesssim 2^{dj}
$$
(which follows from the fact that $\Omega^j_2$ is in a $\sim 2^j$ neighborhood of $\mathcal{S} = \{ \nabla_\eta \phi = 0 \}$, and $\nabla_\xi \nabla_\eta \phi = - \operatorname{Hess} [c] \ (\xi-\eta)$ is not degenerate) and
\begin{equation}
\label{grive}
\left| \nabla_\xi^\alpha \left( m_2(\xi,\eta) \chi_{\Omega^j_2}(\xi,\eta)\chi\left(\frac{\nabla_\xi \psi -t^{-1}z}{r}\right)\right) \right| \lesssim 2^{-|\alpha| j}.
\end{equation}
The two previous inequalities imply
$$ \left|{\mathcal F}_{\xi} \left[m_2(\xi,\eta) \chi_{\Omega^j}(\xi,\eta)\chi\left(\frac{\nabla_\xi \psi -t^{-1}z}{r}\right)\right](a)\right|\lesssim 2^{jd} (1+2^{j} |a|)^{-M}$$
for every integer $M\geq 1$. Hence,
\begin{equation} \left\|{\mathcal F}_{\xi} \left[m_2(\xi,\eta) \chi_{\Omega^j}(\xi,\eta)\chi\left(\frac{\nabla_\xi \psi -t^{-1}z}{r}\right)\right]\right\|_{L^1} \lesssim 1. \label{eq:fo} \end{equation}
Observe now that the phase $\psi$ satisfies $\operatorname{Hess}_\xi [\psi] = \operatorname{Hess}_\xi[a] + \sigma \operatorname{Hess}_\xi [\phi]$; thus it is not degenerate in the $\xi$ variable for $\sigma<\epsilon_0$ provided $\epsilon_0$ is taken sufficiently small, which we ensure. Therefore, Lemma \ref{lem:oscillatory} gives
$$ \left| A_j(\eta,\sigma,y,z)\right| \lesssim t^{-\frac{d}{2}},$$
uniformly in $\eta,\sigma,y,z$.

We claim that $\operatorname{Supp}_\eta A_j$  (for fixed, $\sigma,x,y,z)$) has a measure less than $r^{d-1} 2^j$. Integrating the bound $|A_j| \lesssim t^{-d/2}$ over $(\eta,\sigma) \in \operatorname{Supp}_\eta A_j \times [0,\rho]$ would then give the bound 
$$
|K_{1,1}^j| \lesssim  t^{-d/2} | \operatorname{Supp}_\eta A_j | \rho \lesssim
\frac{2^j r^{d-1} \rho}{t^{d/2}},
$$
which is the desired result.

We now check this claim.

\begin{itemize}
\item Consider first the map
$$ \Lambda\overset{def}{:}(\xi,\eta)\in\R^{2d} \mapsto ( \nabla_\eta \phi(\xi,\eta) , \nabla_\xi \psi(\xi,\eta,\sigma)).$$
A computation gives
$$ D \Lambda(\xi,\eta) = \left(\begin{array}{cc}
                                -\operatorname{Hess}[c](\xi-\eta) &  \operatorname{Hess}_\eta [\phi](\xi,\eta) \\
                                 \operatorname{Hess}_\xi [\psi](\xi,\eta,\sigma) & -\sigma \operatorname{Hess}[c](\xi-\eta)
                               \end{array} \right) $$
which, by the assumption (A1) is invertible for sufficiently small $\sigma$ provided $\epsilon_0$ is chosen sufficiently small. Observe that the support in $(\xi,\eta)$ of $a_j$ is contained in a region where $|\nabla_\eta \phi| \lesssim 2^j \leq r$ and $|\nabla_\xi \psi - t^{-1} z| \lesssim r$. Since $\Lambda$ is a diffeomorphism, it means that the support of $a_j$ in $(\xi,\eta)$ is contained in a ball of radius $\lesssim r$.
\item Next, observe that the support of $a_j$ is contained in a $\sim 2^j$ neighborhood of $\mathcal{R}$. Consider a point $(\xi_0,\eta_0)$ of $\mathcal{R}$, and let $(\xi,\eta) = (\xi_0,\eta_0) + (d\xi,d\eta)$. At the linearized level, $(\xi,\eta)$ is in a $2^j$ neighborhood of $\mathcal{R}$ if and only if
$$
|\nabla_\xi \phi \cdot d\xi| \lesssim 2^j \quad \mbox{and} \quad |\operatorname{Hess}_\eta [\phi] d\eta - \operatorname{Hess} [c] d\xi| \lesssim 2^j.
$$
This implies that $\operatorname{dist}(d\eta,E)\lesssim 2^j$, where $E$ is the $(d-1)$-dimensional linear space $(\operatorname{Hess}_\eta [\phi])^{-1} \operatorname{Hess} [c] (\nabla_\xi \phi)^\perp$. At the nonlinear level, this means that on the support of $a_j$, the variable $\eta$ is restricted to a $\sim 2^j$ neighborhood of a $(d-1)$ dimensional manifold.
\end{itemize}
Combining the two previous points gives the desired estimate:
$$
| \operatorname{Supp}_\eta A_j(\eta,\sigma,y,z) | \lesssim 2^j r^{d-1}.
$$
\end{proof}

\begin{lemma} \label{lem:K_12} For $d\in\{2,3\}$, the kernel $\overline{K}^j_{1,2}$ satisfies the following estimate: for every $\epsilon>0$
$$ \left|\overline{K}^j_{1,2}(y,z)\right| \lesssim \frac{r^{d-1}}{t^{\frac{d}{2}+1+\epsilon}} \frac{1}{(\rho 2^j)^\epsilon} 2^{-j} [2^{-j} + |y|+|z|]^{\epsilon}.$$
\end{lemma}

\begin{proof}
Recall that
\begin{align*}
\overline{K}^j_{1,2} (y,z) = \frac{1}{(2\pi)^d} \int_\rho^{\epsilon_0}  &\int_{\R^{2d}}   e^{it\psi(\xi,\eta,\sigma)}  \\
 &   \partial_{\eta_p}\left[\frac{1}{it\sigma \partial_{\eta_p} \phi} m_2\chi_{\Omega^j} \chi\left(\frac{\nabla_\xi \psi -t^{-1}z}{r}\right) \right]e^{-i(y-z)\eta-iz\xi} \,d\xi \,d\eta\,d\sigma, 
\end{align*}
where we chose the coordinate $\eta_p$ such that $|\nabla_\eta \phi| \simeq |\partial_{\eta_p} \phi| \simeq 2^j$ (which is possible up to some other truncations). Replacing the estimate~\ref{grive} by
\begin{equation}
\left| \nabla_\xi^\alpha \, \partial_{\eta_p} \left[ \frac{1}{t\sigma \partial_{\eta_p} \phi(\xi,\eta)} m_2(\xi,\eta) \chi_{\Omega^j_2}(\xi,\eta)\chi\left(\frac{\nabla_\xi \psi -t^{-1}z}{r}\right)\right] \right| \lesssim \frac{2^{-(|\alpha|+2) j}}{t \sigma},
\end{equation}
we can proceed as in the proof of Lemma~\ref{lem:K_11}, use the stationary phase Lemma \ref{lem:oscillatory} to estimate the oscillatory integral in $\xi$ and get
\begin{equation} |{\overline K}^j_{1,2}(y,z)| \lesssim \int_\rho^{\epsilon_0} \frac{r^{d-1}2^j}{t^{\frac{d}{2}+1} 2^{2j}} \frac{d\sigma}{\sigma} \simeq \frac{r^{d-1}}{t^{\frac{d}{2}+1}} |\log(\rho)| 2^{-j}. \label{eq:int1} \end{equation}
A second integration by parts in $\eta$ using the identity $\frac{1}{i\sigma t \partial_{\eta_p} \phi} \partial_{\eta_p} e^{it \sigma \phi}$ gives by the same token
\begin{equation} |{\overline K}^j_{1,2}(y,z)| \lesssim 2^{-j} [2^{-j} + |y|+|z|] \int_\rho^{\epsilon_0} \frac{r^{d-1}2^j}{t^{\frac{d}{2}+2} 2^{2j}} \frac{d\sigma}{\sigma^2} \simeq \frac{r^{d-1}}{t^{\frac{d}{2}+2}} \frac{1}{\rho 2^j} 2^{-j} [2^{-j} + |y|+|z|]. \label{eq:int2} \end{equation}
Interpolating between (\ref{eq:int1}) and (\ref{eq:int2}) gives for every $\epsilon>0$ 
\begin{align*}
 |{\overline K}^j_{1,2}(y,z)| & \lesssim 2^{-j}[2^{-j} + |y|+|z|]^{\epsilon} \int_\rho^{\epsilon_0} \frac{r^{d-1}2^j}{t^{\frac{d}{2}+1+\epsilon} 2^{(1+\epsilon)j}} \frac{d\sigma}{\sigma^{1+\epsilon}} \\
 & \lesssim \frac{r^{d-1}}{t^{\frac{d}{2}+1+\epsilon}} \frac{1}{(\rho 2^j)^\epsilon} 2^{-j} [2^{-j} + |y|+|z|]^{\epsilon}.
\end{align*}
\end{proof}

\begin{lemma} \label{lem:E} For $d\in\{2,3\}$, as soon as $t\rho\geq 2^{ja}$ (for some parameter $a < 0$), the bilinear operator ${\mathcal E}_j$ satisfies: for every $s>\frac{3}{2}$ there exists $\nu>0$ such that
$$ \|{\mathcal E}_j \|_{L^{2,s} \times L^{2,s} \to L^\infty} \lesssim t^{-1-\frac{d}{2}} (r2^{-j})^{\frac{d-1}{d}} 2^{j\nu} .$$
\end{lemma}

\begin{proof}
Recall that 
$$ {\mathcal E}_j(f,g) := i \left( T_{R^j_{1,2}}(x_pf,g) + T_{R^j_{1,2}}(f,x_p g)\right),$$
with the kernel 
\begin{align*} 
 \lefteqn{R^j_{1,2}(y,z) \overset{def}{=}} & & \\
 & & \frac{1}{(2\pi)^d} \int_\rho^{\epsilon_0}  \int_{\R^{2d}}  e^{it\psi(\xi,\eta,\sigma)} \left[\frac{m_2 \chi_{\Omega^j_2}}{it\sigma \partial_{\eta_p} \phi} \, \chi\left(\frac{\nabla_\xi \psi -t^{-1}z}{r}\right)  \right]  e^{-i(y-z)\eta-iz\xi} \,d\xi \,d\eta\,d\sigma, 
\end{align*}
where we chose the coordinate $\eta_p$ such that $|\nabla_\eta \phi| \simeq |\partial_{\eta_p} \phi| \simeq 2^j$.
First, an argument similar to the proof of Lemma \ref{lem:K_12} yields that $R^j_{1,2}$ satisfies the pointwise bound
$$ \left|R^j_{1,2}(y,z) \right| \lesssim \frac{r^{d-1}}{t^{\frac{d}{2}+1+\epsilon}} \frac{1}{(\rho 2^j)^\epsilon} [2^{-j} + |y|+|z|]^{\epsilon},$$
for every $\epsilon>0$.
This means that for every $\epsilon>0$ and $s>\frac{d}{2}+\epsilon$,
\begin{align} \| T_{R^j_{1,2}} \|_{L^{2,s} \times L^{2,s} \to L^\infty} \lesssim \frac{r^{d-1}}{t^{\frac{d}{2}+1+\epsilon}} \frac{1}{(\rho 2^{j})^\epsilon} . \label{eq:E1} \end{align}
We will now derive another estimate by viewing the problem in Fourier space, and finally interpolate between the two estimates. More specifically, writing first
$$
\chi \left(\frac{\nabla_\xi \psi -t^{-1}z}{r} \right) = \frac{1}{(2\pi)^{d/2}} \int_{\mathbb{R}^d} r^d \widehat{\chi}(r \theta) e^{i\theta (\nabla_\xi a + \sigma \nabla_\xi \phi + t^{-1} x - t^{-1} z)}\,d\theta,
$$
one obtains the formula for $T_{R^j_{1,2}}$
\begin{equation}
\label{rougequeue}
\begin{split}
& T_{R^j_{1,2}}(f,g)(x) =\\
& \qquad \quad \frac{1}{(2\pi)^{d/2}} \int_{\R^d} r^d \widehat{\chi}(r \theta) \left[ e^{i t^{-1} \theta x} t^{-1} \int_{\rho t}^{\epsilon_0 t} T_{m_{s,j,\theta}} ( e^{isb(D)} f, e^{isc(D)} e^{-i t^{-1} \theta \cdot} g)\, ds\right]\, d\theta,
\end{split}
\end{equation}
with the symbol
$$ 
m_{s,j,\theta}(\xi,\eta) \overset{def}{=} e^{i \theta (\nabla_\xi a + \sigma \nabla_\xi \phi)} \frac{1}{is \partial_{\eta_p} \phi(\xi,\eta)} (m_2 \chi_{\Omega^j_2})(\xi,\eta).
$$
Then this symbol has the following regularity (uniformly in $s,\theta$):
$$ \left|\nabla_\eta^\alpha \nabla_\xi^\beta m_{s,j,\theta}(\xi,\eta)\right| \lesssim \frac{1}{s2^j} 2^{-j(|\alpha|+|\beta|)}$$
since $r\geq 2^j$ and one can assume $t|\theta| \lesssim r^{-1}$ due to the fast decay of $\widehat{\chi}$. Moreover $m_{s,j,\theta}$ is supported on $\Omega^j_2$, which is included into a $2^j$-neighborhood of the translated $d$-dimensional manifold ${\mathcal S}:=\{ \nabla_\eta \phi=0\}$. This sub-manifold may be parametrized by the $\eta$-variable or $(\xi-\eta)$-variable (indeed $\operatorname{Hess}_\eta [\phi]$ and $(\nabla_\xi+\nabla_\eta) \nabla_\eta \phi$ are non-degenerate on ${\mathcal S}$). So $\Omega^j_2$ may be covered by a collection of $2^j$-balls whose projections on $\eta$ and $\xi-\eta$ are almost disjoint. Each of these elementary bilinear operators (obtained by restricting the symbol on one of the balls) is bounded from $L^2 \times L^2$ to $L^1$ with a norm $\lesssim \frac{1}{s2^j}$. By the Cauchy-Schwartz inequality and orthogonality between the elementary operators, this same bound can be obtained for $T_{m_{s,j,\theta}}$ (we refer the reader to \cite{BG2} for such an argument in the one dimensional case). In other words, uniformly in $\theta$:
\begin{align} \|T_{m_{s,j,\theta}} \|_{L^2 \times L^2 \to L^1} \lesssim \frac{1}{s2^j}. \label{eq:221} \end{align}
Similarly, one integration by parts in $\eta$ in
\begin{align*} 
\lefteqn{T_{m_{s,j,\theta}} ( e^{isb(D)} f, e^{isc(D)} e^{-i t^{-1} \theta \cdot} g)(x)=} & & \\
 & &  \int\int  e^{ix\xi} m_{s,j,\theta}(\xi,\eta) e^{is(b(\eta)+c(\xi-\eta))}\widehat{f}(\eta) \widehat{g}(\xi-\eta+t^{-1} \theta) d\eta d\xi,
\end{align*}
yields the following bound
\begin{align*} \|T_{m_{s,j,\theta}} ( e^{isb(D)} \cdot , e^{isc(D)} e^{-i t^{-1} \theta} \cdot ) \|_{L^{2,1} \times L^{2,1} \to L^1} \lesssim \frac{1}{s^2 2^{3j}} \end{align*}
since on $\Omega^j_2$, $|\nabla_\eta \phi(\xi-\theta,\eta)| \simeq 2^j$.
So by interpolation with (\ref{eq:221}): for every $\gamma>0$ (uniformly with respect to $s,\theta$)
\begin{align*} \|T_{m_{s,j,\theta}} \|_{L^{2,\gamma} \times L^{2,\gamma} \to L^1} \lesssim \frac{1}{s^{1+\gamma} 2^{(1+2 \gamma)j}}.\end{align*}
Now coming back to the operator $T_{R^j_{1,2}}$ via the formula~\eqref{rougequeue}, the dispersive estimates on the propagators give
\begin{align}
 \|T_{R^j_{1,2}}\|_{L^{2,\gamma} \times L^{2,\gamma} \to L^\infty} & \lesssim  t^{-1} \int_{\rho t}^{\epsilon_0 t} (t-s)^{-\frac{d}{2}} \frac{1}{s^{1+\gamma} 2^{(1+2 \gamma)j}}  ds \nonumber \\
 & \lesssim t^{-1-\frac{d}{2}}  2^{-(1+2 \gamma)j} (t\rho)^{-\gamma}. \label{eq:E2}
\end{align}
By interpolation between (\ref{eq:E1}) and (\ref{eq:E2}) and using the assumption that $t\rho\geq 2^{ja}$ for some parameter $a<0$, we deduce that for every $s>\frac{1}{2}$ there exists $\nu>0$ with
\begin{align}
 \|T_{R^j_{1,2}}\|_{L^{2,s} \times L^{2,s} \to L^\infty} & \lesssim  t^{-1-\frac{d}{2}}  (r2^{-j})^{\frac{d-1}{d}} 2^{j\nu},
\end{align}
which allows us to conclude.
\end{proof}

\begin{lemma} \label{lem:K_2} Assume $d \in \{2,3\}$. The kernel $K^j_2$ satisfies the following estimate: for any $N\geq 2$,
$$ \left|K^j_2(y,z)\right| \lesssim (t2^{j}r)^{1-N} \frac{2^{dj}}{t} = (t2^{j}r)^{-N} 2^{j(d+1)}r.$$
\end{lemma}

\begin{proof} Recall first that
\begin{equation*}
\begin{split}
& K^j_2(y,z)= \\
&\quad \frac{1}{(2\pi)^d}  \int_0^{\epsilon_0}  \int_{\R^{2d}}  e^{it\psi(\xi,\eta,\sigma)} (m_2\chi_{\Omega^j_2})(\xi,\eta) e^{-i(y-z)\eta-iz\xi}  \left[1-\chi\left(\frac{\nabla_\xi \psi -t^{-1}z}{r}\right)\right]\,d\xi \,d\eta\,d\sigma.
\end{split}
\end{equation*}
To estimate $K^j_2$, we will integrate by parts in $\xi$ using the identity 
$$
\frac{1}{it (\partial_{\xi_q} \psi - t^{-1} z_q)} \partial_{\xi_q} e^{it(\psi - t^{-1}z\xi)} = e^{it(\psi - t^{-1}z\xi)},
$$
where we take the coordinate $q$ such that $|\partial_{\xi_q} (\psi - t^{-1}z\xi)| \sim |\nabla_\xi (\psi - t^{-1}z\xi)|$ (introducing cutoff functions if need be).
Using the estimate
$$ \left|\nabla_\xi^\alpha \left[ m_2(\xi,\eta) \chi_{\Omega^j_2}(\xi,\eta) e^{-i(y-z)\eta}  \left[1-\chi\left(\frac{\nabla_\xi \psi-t^{-1}z}{r}\right)\right]\right] \right| \lesssim 2^{-j|\alpha|},$$
and the fact that $r \geq 2^j$, we get after $N$ integration by parts the bound
\begin{equation}
\label{bruant}
\left|K^j_2(y,z)\right| \lesssim \int_0^{\epsilon_0}  \int_{\R^{2d}}  \left(\frac{1}{t|\nabla_\xi \psi - t^{-1}z| 2^{j}} \right)^N {\bf 1}_{\Omega^j_2}(\xi,\eta) {\bf 1}_{|\nabla_\xi \psi-t^{-1}z|\geq r} \,d\xi \,d\eta\,d\sigma.
\end{equation}
Now observe that, for $(\xi,\eta)$ fixed, 
\begin{equation}
\label{bergeronnette}
\int_0^{\epsilon_0} \frac{1}{|\nabla_\xi \psi - t^{-1}z|^N} {\bf 1}_{|\nabla_\xi \psi-t^{-1}z|\geq r}\,d\sigma \lesssim r^{1-N}.
\end{equation}
Indeed, $\nabla_\xi \psi - t^{-1}z = \nabla_\xi a + \sigma \nabla_\xi \phi + X - t^{1}{z}$. By the assumption (A2), $\nabla_\xi \phi$ is not zero in a neighbourhood of $\mathcal{R}$. Therefore, for $(\xi,\eta)$ fixed, $\nabla_\xi \psi - t^{-1}z = \sigma u + v$, with $u,v \in \mathbb{R}^d$ and $|u|\gtrsim 1$. It is now easy to show that 
$$
\int_0^{\epsilon_0} \frac{1}{|\sigma u + v|^N} 
{\bf 1}_{|\sigma u + v|\geq r}\,d\sigma \lesssim r^{1-N},
$$
hence the inequality~(\ref{bergeronnette}).

Coming back to~(\ref{bruant}), this gives
\begin{equation*}
\left|K^j_2(y,z)\right| \lesssim \int_{\R^{2d}} {\bf 1}_{\Omega^j_2}(\xi,\eta) 
\frac{1}{(t2^j)^N} r^{1-N} \,d\xi \,d\eta \lesssim (t2^{j})^{-N} r^{1-N} |\Omega^j_2| \lesssim (t2^{j}r)^{1-N} \frac{2^{dj}}{t}.
\end{equation*}
\end{proof}

\section{Application to water waves}

\label{sectionww}

\subsection{A toy model}

Consider the water waves system in dimension 3, with infinite depth, gravity and surface tension. To be more specific: consider an irrotational fluid, governed by Euler's equation,
with a free surface; assume that the only forces acting on the fluid are the gravity $g$ and surface tension $\sigma$.
The free surface is given by a graph parameterized by the height function $h$: it is $\{ (x,h(x)), x \in \mathbb{R}^2 \}$. The domain occupied by the fluid
is infinitely deep: it is given by $\{ (x,z) \in \mathbb{R}^2 \times \mathbb{R} \;\mbox{such that}\; z \leq h(x) \}$. Since the fluid is irrotational and incompressible, its 
velocity is given by the gradient of a harmonic function: $v = \nabla \widetilde{\psi}$. Finally, $\widetilde{\psi}$ is fully determined by its value on the free surface, which we denote
by $\psi(x) \overset{def}{=} \widetilde{\psi}(x,h(x))$.

We will describe this system by the ``Zakharov variables'' $(h,\psi)$. Our interest lies in stability of the equilibrium state given by $(h,\psi) = (0,0)$. It was established
in~\cite{GMS1}~\cite{Wu} in the case $g>0$, $\sigma = 0$, and in~\cite{GMS2} in the case $g=0$, $\sigma>0$; to be more specific, it is proved in these papers that, for data sufficiently close
to $(0,0)$ in a sufficiently strong topology, the resulting solution is global, and scatters in an $L^2$-type space as $t \rightarrow \infty$. 

The case $g,\sigma>0$ is considerably more difficult, for two main reasons: absence of scaling invariance (which is clear), and more intricate resonant structure (which will be described shortly).
Our aim in this section is to initiate the study of the case $g,\sigma>0$, at a fairly heuristic level. To be more precise, we will study the second iterate (in an iterative resolution scheme)
of a toy model; this toy model is obtained by dropping in the expansion of the water waves equation around $(0,0)$ all terms which are cubic or of higher order. 
Upon setting $\Lambda = |D|$, it reads (see Sulem and Sulem~\cite{SuSu})
$$
\left\{ \begin{array}{l} \partial_t h = \Lambda \psi - \nabla \cdot( h \nabla \psi) - \Lambda (h \Lambda \psi) \\
         \partial_t \psi = - g h + \sigma \Delta h - \frac{1}{2} |\nabla \psi|^2 + \frac{1}{2} |\Lambda \psi|^2 
        \end{array}
\right.
$$
Setting
$$
\widetilde{h} \overset{def}{=} p(D) h \qquad \mbox{with} \qquad p(D) = \sqrt{ \frac{g - \sigma \Delta}{\Lambda} },
$$
it becomes
$$
\left\{ \begin{array}{l} \partial_t \widetilde{h} = \Lambda p(D) \psi + T_{m_1} (\widetilde{h},\psi) \\
         \partial_t \psi = (-g + \sigma \Delta) p(D)^{-1} \widetilde{h} + T_{m_2} (\psi,\psi)
\end{array} \right.
$$
with
$$
\left\{ \begin{array}{l} m_1(\xi,\eta) = (2\pi)^{d/2} \frac{p(\xi)}{p(\eta)} ( \xi \cdot (\xi-\eta) - |\xi| |\xi-\eta| ) \\
         m_2(\xi,\eta) = \frac{ (2\pi)^{d/2}}{2} \left[ \eta \cdot (\xi-\eta) + |\eta| |\xi-\eta| \right].
        \end{array}
\right.
$$
Finally, introduce the complex variable
$$
u = \widetilde{h} + i \psi
$$
for which the equation reads
$$
i \partial_t u - \tau(D) u = T_{m_{++}} (\bar u, \bar u) + T_{m_{--}} (u,u) + T_{m_{+-}} (\bar u,u) + T_{m_{-+}} (u,\bar u)
$$
where
$$
\tau(x) \overset{def}{=} \sqrt{|x|( g + \sigma |x|^2)} \quad \mbox{and thus} \quad \tau(D) = \sqrt{\Lambda ( g - \sigma \Delta)}
$$
and $m_{\pm,\pm}$ are linear combinations of $m_1$ and $m_2$. We will also denote
$$
\mbox{if $\lambda>0$,} \qquad \bar \tau(\lambda) \overset{def}{=} \sqrt{\lambda( g + \sigma \lambda^2)} \quad \mbox{so that} \quad \tau(x) = \bar \tau(|x|).
$$

\subsection{The linear problem} We discuss here the linear problem
$$
i \partial_t u = \tau(D) u.
$$
A small computation reveals that, for $\lambda > 0$,
$$
\bar \tau'(\lambda) = \frac{g + 3 \sigma \lambda^2}{2 \sqrt{ g\lambda + \sigma \lambda^3}} \quad \mbox{and} \quad \bar \tau''(\lambda) 
= \frac{\frac{3}{2} \sigma^2 \lambda^4 + 3 g \sigma \lambda^2 - \frac{1}{2} g^2}{2 ( g \lambda + \sigma \lambda^3)^{3/2}}.
$$
It is then easy to see that the graph of $\bar \tau$ switches from concave to convex at
$$
\lambda_0 \overset{def}{=} \sqrt{\frac{(2\sqrt{3}-3)}{3} \frac{g}{\sigma}}.
$$
In other words, the Hessian of $\tau$ is non-degenerate if and only if $\lambda \neq \lambda_0$. Therefore, the dispersive estimates~(\ref{dispersive}) hold for functions whose Fourier transform is compactly supported away from the frequencies $\xi$ of size $|\xi| = \lambda_0$, whereas the decay if $|\xi|=\lambda_0$ should be slower. The exact rate of decay and number of derivatives needed has been obtained by Spirn and Wright~\cite{SW}, who showed that
$$ \|e^{it \tau(D)} \|_{B^{3/4}_{1,1} \to L^\infty} \lesssim t^{-\frac{5}{6}}$$
($B^{3/4}_{1,1}$ being the classical non-homogeneous Besov space).
We point out (as we will observe later) that even if the phase may be degenerate, it will be non-degenerate close to the resonant set. So up to some extra cut-off, to study the phenomenon of space-time resonances, this degeneracy is not relevant.

\subsection{Discussion of quadratic resonances}

Four different interactions appear in the above equation $++$, $+-$, $-+$ and $--$. The third one can be obtained from the second by symmetry, thus we will ignore it.
To study the space ($\mathcal{S}_{\pm \pm}$), time ($\mathcal{T}_{\pm \pm}$), and space-time ($\mathcal{R}_{\pm \pm}$) resonance sets for these various interactions, let us introduce the corresponding phases
$$
\phi_{\pm \pm} (\xi,\eta) = \tau(\xi) \pm \tau(\eta) \pm \tau(\xi-\eta).
$$
In the following, we characterize the space-time resonance sets in the cases $++$, $+-$ and $--$. It turns out that a function, which will be denoted $\alpha$, will
play a key role; we will discuss its properties in some detail.

\subsubsection{The $++$ interaction} This case is trivial since $\mathcal{T}_{++}=\{ 0,0 \}$.

\subsubsection{The function $\alpha$} \label{subsec:alpha} Recall that $\bar \tau$ switches from concave to convex at $\lambda_0$. Furthermore, it is easy to see that $\lim_{0} \bar \tau' = + \infty$ whereas $\lim_{+\infty} \bar \tau' = + \infty$.

It is therefore clear that any given value larger than $\bar \tau'(\lambda_0)$ is reached by $\bar \tau'$ at exactly two points. Let us denote $\alpha$ for the function which exchanges these two points:
$$
\mbox{if $\mu>0$,} \qquad \alpha(\mu) \overset{def}{=} \left\{ \begin{array}{ll} \lambda_0 & \mbox{if $\mu=\lambda_0$} \\
                                                                \nu > 0 \; \mbox{such that $\nu \neq \mu$ and $\bar \tau'(\nu) =\bar \tau'(\mu)$} & \mbox{if $\mu \neq \lambda_0$}.
                                                               \end{array}
\right.
$$
It has the following properties
\begin{itemize}
\item \vsp As $\lambda \rightarrow 0$, $\alpha(\lambda) = \frac{g}{9\sigma \lambda} + O(1)$.
\item \vsp For any $\lambda>0$, $\alpha'(\lambda)<0$.
\item \vsp As $\lambda \rightarrow \infty$, $\alpha(\lambda) = \frac{g}{9\sigma \lambda} + O\left( \frac{1}{\lambda^2} \right)$.
\end{itemize}
These equivalents are a consequence of $\bar \tau'(\lambda) \simeq_{\infty} \frac{3}{2}\sqrt{\sigma \lambda}$ and $\bar \tau'(\alpha) \simeq_0 \frac{1}{2} \sqrt{\frac{g}{\alpha}}$.

\subsubsection{The $+-$ interaction} We consider here $\phi_{+-}(\xi,\eta) = \tau(\xi) + \tau(\eta) - \tau(\xi-\eta)$. The space resonant set is
$$
\mathcal{S}_{+-} = \left\{ (\xi,\eta) \; \mbox{such that $\nabla \tau(\eta) = -\nabla \tau(\xi-\eta)$} \right\}.
$$
By rotational symmetry, we can assume that $\eta = \lambda e_0$, for a fixed direction $e_0$ and $\lambda>0$.
Then it follows from the above formula that $(\xi,\eta)$ belongs to $\mathcal{S}_{+-}$ if
$$
\xi = (\lambda - \mu) e_0 \qquad \mbox{with} \quad \left\{ \begin{array}{l} \mbox{either $\mu = \lambda$} \\ \mbox{or $\mu = \alpha(\lambda)$.} \end{array} \right.
$$
The first possibility, namely $\mu = \lambda$, leads to $\xi = 0$. But it is clear that $\phi_{+-}(0,\eta)=0$, thus $(0,\eta) \in \mathcal{R}_{+-}$ for any $\eta \in \mathbb{R}^2$. We now turn to the second possibility, namely $\mu = \alpha(\lambda)$. Let
$$
f(\lambda) \overset{def}{=} \bar \tau (|\lambda - \alpha(\lambda)|) + \bar \tau(\lambda) - \bar \tau(\alpha(\lambda)).
$$
Matters reduce now to finding zeros of $f$, since $f(\lambda)=0$ if and only if $(\xi,\eta) = ((\lambda - \alpha(\lambda))e_0,\lambda e_0) \in \mathcal{R}_{+-}$.
A small computation gives
$$
f'(\lambda) = \left\{ \begin{array}{ll} (1-\alpha'(\lambda))(\tau'(\lambda - \alpha(\lambda))+\tau'(\lambda)) & \mbox{if $\lambda>\lambda_0$} \\
                       (1-\alpha'(\lambda))(-\tau'(\alpha(\lambda)-\lambda)+\tau'(\lambda)) & \mbox{if $\lambda<\lambda_0$} \\
                      \end{array}
\right.
$$
Denoting $\lambda_1$ for the unique solution of $\alpha(\lambda_1) = 2 \lambda_1$, clearly $\lambda_1 < \lambda_0$. It is easy to deduce from the above that
$f'>0$ on $(0,\lambda_1) \cap (\lambda_0,\infty)$, whereas $f'<0$ on $(\lambda_1,\lambda_0)$. It is now possible to fully describe $f$:
\begin{itemize}
\item \vsp As $\lambda \rightarrow 0$, $f(\lambda) = O(\sqrt{\lambda})$ (follows from the properties of $\alpha$).
\item \vsp $f$ is increasing on $(0,\lambda_1)$, thus $f(\lambda_1)>0$.
\item \vsp $f$ is decreasing on $(\lambda_1,\lambda_0)$, down to $f(\lambda_0)=0$.
\item \vsp $f$ is increasing on $(\lambda_0,\infty)$, with $f(\lambda) \rightarrow \infty$ as $\lambda \rightarrow \infty$.
\end{itemize}
Thus, $f$ only vanishes at $\lambda_0$; this implies that $(0,\eta) \in \mathcal{R}_{+-}$ if $|\eta| = \lambda_0$. But we already saw above that $(0,\eta) \in \mathcal{R}_{+-}$ for any $\eta$. The general conclusion is that
$$
\mathcal{R}_{+-} = \{ (0,\eta) \;,\;\eta \in \mathbb{R}^d \}.
$$

\subsubsection{The $--$ interaction}We consider now $\phi_{--}(\xi,\eta) = \tau(\xi) - \tau(\eta) - \tau(\xi-\eta)$. The space resonant set is
$$
\mathcal{S}_{--} = \left\{ (\xi,\eta) \; \mbox{such that $\nabla \tau(\eta) = \nabla \tau(\xi-\eta)$} \right\}.
$$
Once again, we can assume that $\eta = \lambda e_0$, for a fixed direction $e_0$ and $\lambda>0$. Then $(\xi,\eta)$ belongs to $\mathcal{S}_{--}$ if
$$
\xi = (\lambda + \mu) e_0 \qquad \mbox{with} \quad \left\{ \begin{array}{l} \mbox{either $\mu = \lambda$} \\ \mbox{or $\mu = \alpha(\lambda)$.} \end{array} \right.
$$
Let us start with the first subcase, namely $\mu = \lambda$ ie $\xi = 2 \lambda e_0$. Then $(\xi,\eta) \in \mathcal{R}_{--}$ if and only if
$$
 0 = \phi_{--}(2\lambda e_0,\lambda e_0) = \bar \tau (2 \lambda) - 2 \bar \tau(\lambda) 
= \frac{-2g\lambda + 4 \sigma \lambda^3}{\sqrt{2g\lambda + 8 \sigma \lambda^3} + \sqrt{4g\lambda + 4 \sigma \lambda^3}} 
$$
or equivalently
$$ \lambda = \lambda_2 \overset{def}{=} \sqrt{\frac{g}{2\sigma}}$$
(notice that $\lambda_2 > \lambda_0$). This corresponds to the space-time resonant points $(\xi,\eta) = (2 \eta,\eta)$, with $\eta \in \mathbb{R}^2$ and $|\eta|=\lambda_2$. This type of space-time resonance is well-known in the literature and has been called the second harmonic resonance~\cite{SuSu,MN}.

There remains to examine the second subcase, namely $\mu = \alpha(\lambda)$ ie $\xi = (\lambda + \alpha(\lambda)) e_0$. Such a point belongs to $\mathcal{R}_{--}$ 
if and only if $f(\lambda)=0$, with
$$
f(\lambda) = \bar \tau (\lambda + \alpha(\lambda)) - \bar \tau(\lambda) - \bar \tau(\alpha(\lambda)).
$$
We claim that $f(\lambda)<0$ for every $\lambda$.
Let us now check this fact : for convenience, we write $\alpha=\alpha(\lambda)$. First we claim that for any $\lambda>0$,
\begin{equation}
 9 (1-c_0)^2  \sigma \lambda \alpha \leq (1+3c_0)^2 g,  \label{eq:la}
\end{equation}
where $c_0\overset{def}{=}\frac{(2\sqrt{3}-3)}{3}\simeq 0.15$ is such that $\lambda_0^2 \sigma = c_0 g$.

Let us check this property. By symmetry $\lambda \leftrightarrow \alpha$ we may assume that $\lambda \geq \lambda_0$. On $[\lambda_0,\infty)$, we have
\begin{equation} \bar \tau'(\lambda) \geq \ell_1(\lambda)\overset{def}{=}\frac{3}{2} (1-c_0) \sqrt{\sigma \lambda} \label{eq:eeee} \end{equation}
since
$$ 4\left[ (\bar \tau'(\lambda))^2 - \frac{9}{4} (1-c_0)^2 \sigma \lambda \right] = \frac{(g+3\sigma  \lambda^2)^2-9(1-c_0)^2 \sigma \lambda (g\lambda+\sigma \lambda^3)}{g\lambda+\sigma \lambda^3} > 0$$
and the numerator can be seen as a polynom in $\lambda^2$ of order 2 with a negative discriminant.
If $\alpha \in (0,\lambda_0]$, we have
\begin{equation} \bar \tau'(\alpha) \leq \ell_2(\alpha)\overset{def}{=} (1+3c_0) \frac{\sqrt{g}}{2\sqrt{\alpha}}, \label{eq:eeee2} \end{equation}
since 
$$ \bar \tau'(\alpha) \leq \frac{g+3 \sigma \lambda_0^2}{2\sqrt{g\alpha}}.$$
By definition of the function $\alpha=\alpha(\lambda)$, using (\ref{eq:eeee}) and (\ref{eq:eeee2}) it follows that
$$ \alpha(\lambda) \leq \ell_2^{-1}(\ell_1(\lambda)),$$
hence (\ref{eq:la}).

We now come back to the proof of $f(\lambda)<0$. Again, we may assume that $\lambda \geq \lambda_0$. Then using the convexity properties of $\bar \tau$ it comes
$$ \bar \tau(\alpha+\lambda)- \bar \tau(\lambda) \leq \alpha \bar \tau '(\alpha+\lambda)$$
and we claim that $\alpha \bar \tau '(\alpha+\lambda) < \tau(\alpha)$ (which implies the desired result). By squaring the formulas for $\bar \tau$ and $\bar \tau '$, this is equivalent to
\begin{align}
   \alpha(g+3\sigma(\alpha+\lambda)^2)^2 < 4 (\alpha+\lambda) (g+\sigma \alpha^2) (g+\sigma (\alpha+\lambda)^2). \label{eq:res}
\end{align}
First, we have
\begin{equation} 9 \sigma^2 \alpha (\alpha+\lambda)^4 < 4 \sigma (g+\sigma \alpha^2) (\alpha + \lambda)^3, \label{eq:1} \end{equation}
which is equivalent to 
$$  \sigma \alpha (9\lambda+5\alpha) < 4 g $$
and also is a consequence of $\alpha\leq \lambda_0$, (\ref{eq:la}) and
$$ \frac{(1+3c_0)^2}{(1-c_0)^2} + 5c_0 \simeq 3.77 < 4.$$  
 
Moreover we claim that
\begin{equation}
 6 \alpha \sigma g (\alpha+\lambda)^2 <  3(\alpha+\lambda) g^2.  \label{eq:2}
\end{equation}
Indeed this is implied by (using $\alpha\leq \lambda_0$)
$$  2 \sigma \lambda_0^2 +2 \sigma \alpha \lambda <  g,$$
which, using once again~(\eqref{eq:la}) is a consequence of $2c_0 +\frac{2(1+3c_0)^2}{9(1-c_0)^2} <1$.

Consequently, combining (\ref{eq:1}) and (\ref{eq:2}) allows us to conclude to (\ref{eq:res}) by expanding it, which ends the proof of $f(\lambda)<0$. Therefore,
$$
\mathcal{R}_{--} = \{(2 \eta,\eta)\;\mbox{with $\eta \in \mathbb{R}^2$ and $|\eta|=\lambda_2$}\}.
$$

\subsubsection{Conclusion} Let us summarize our findings:
\begin{itemize}
\item[(a)] For the $++$ interaction, $\mathcal{R}_{++} = \{(0,0)\}$.
\item[(b)] For the $+-$ interaction, $\mathcal{R}_{+-} = \{(0,\eta)\;\mbox{with $\eta \in \mathbb{R}^2$}\}=\{0\} \times \mathbb{R}^2$.
\item[(c)] For the $--$ interaction, $\mathcal{R}_{--} = \{(2 \eta,\eta)\;\mbox{with $\eta \in \mathbb{R}^2$ and $|\eta|=\lambda_2$}\}$
\end{itemize}
A brief look at the interaction symbols $m_1$ and $m_2$ reveals that $m_1(0,\eta) = m_2(0,\eta) = 0$ for any $\eta \in \mathbb{R}^2$;
this null form structure was already noticed and used in~\cite{GMS1, GMS2}. The analysis in these papers leads us to think that the resonances (a) and (b) will be effectively 
cancelled. However, this is not the case for (c): $m_1(2 \eta,\eta)$ vanishes, but not $m_2(2\eta,\eta)$ ! These space-time resonances will play a role in the dynamics. 

Let us check that this space-time resonant interaction ($--$ interaction with $|\eta| = \lambda_2$ and $\xi=2\eta$) is of the generic type studied here, namely that the assumptions~(\ref{pingouin}), (A1), (A2), and (A3) are satisfied.

\begin{itemize}
\item \eqref{pingouin} is satisfied for any frequency $\xi$ such that $|\xi| \neq \lambda_0$. This is the case if $(\xi,\eta)$ is sufficiently close to $\mathcal{R}_{--}$ since $\lambda_2 \neq \lambda_0$ and $\lambda_0 \neq 2 \lambda_2$.
\item (A1) is satisfied since for an arbitrary point $(2\eta,\eta)$ of $\mathcal{R}_{--}$,
$$\left[ \nabla_\xi \phi_{--} \right] (2\eta,\eta) = \nabla \tau (2\eta) - \nabla \tau(\eta) \neq 0.$$
\item To see that (A2) is satisfied, observe that for a point $(2\eta,\eta)$ of $\mathcal{R}_{--}$,
$$\operatorname{Hess}_\eta [\phi_{--}] (2\eta,\eta) = -2 \operatorname{Hess} [\tau](\eta).$$ 
Furthermore, $ \operatorname{Hess} [\tau](\eta)$ is positive: we will see in Section~\ref{appendice} that this is the case since $\bar \tau'(\lambda_2) > 0$ and $\bar \tau''(\lambda_2) > 0$ (which is a consequence of $\lambda_2 > \lambda_0$).
\item Finally, we let the reader to check that (A3)  is also satisfied (indeed, this is a consequence of easy computations using the Property \ref{prop} and definition of $\lambda_2$).
\end{itemize}

\subsection{Decay properties} Combining the previous observations with Theorem~\ref{theoremeasymptotique}, this makes the following behavior plausible for a solution $u$ of the full water-wave system:
\begin{itemize}
\item Restricting $u$ to frequencies away from $\lambda_0$ and $2 \lambda_2$ (by using a smooth Fourier multiplier vanishing there), one observes an $L^\infty$ decay $\sim \frac{1}{t}$.
\item Close to the frequency $\lambda_0$, the $L^\infty$ decay is of the order $\sim \frac{1}{t^{5/6}}$.
\item Finally, close to the frequency $\lambda_2$, the $L^\infty$ decay is of the order $\sim \frac{\log t}{t}$.
\end{itemize}

\section{Appendix: genericity of (A3) in the isotropic case} \label{appendice}

We would like here to precise some computations in the case where the dispersive relations $a,b,c$ are given by radial functions.

First an easy computation gives the following: if $\tau(\xi)= \bar \tau(|\xi|)$ with $\bar \tau$ smooth,
$$  \operatorname{Hess} [\tau](\xi) = \left( \frac{\bar \tau ''(|\xi|)}{|\xi|^2} - \frac{\bar \tau '(|\xi|)}{|\xi|^3} \right)  \xi. \xi^t + \frac{\bar \tau '(|\xi|)}{|\xi|} \textrm{Id}$$ such that
$$  \langle \operatorname{Hess} [\tau] (\xi)u,u\rangle = \left( \frac{\bar \tau ''(|\xi|)}{|\xi|^2} - \frac{\bar \tau '(|\xi|)}{|\xi|^3} \right) |\langle u,\xi\rangle|^2 + \frac{\bar \tau '(|\xi|)}{|\xi|} |u|^2,$$
or equivalently
$$  \langle \operatorname{Hess} [\tau ](\xi)u,u\rangle = \frac{\bar \tau ''(|\xi|)}{|\xi|^2} |\langle u,\xi \rangle |^2 + \frac{\bar \tau '(|\xi|)}{|\xi|} |u^{\perp\xi}|^2,$$
where $u^{\perp \xi}= u-\langle u,\xi\rangle \frac{\xi}{|\xi|^2}$ is the projection of $u$ on $\xi^\perp$.
Consequently, we deduce the following property:

\begin{property} For $\xi\neq 0$, the Hessian matrix $\operatorname{Hess} [\tau] (\xi)$ is non-degenerate if and only if $\bar \tau '(|\xi|) \neq 0$ and $\bar \tau ''(|\xi|)\neq 0$.
Moreover, $\operatorname{Hess} [\tau] (\xi)$ is a linear combination of the identity and of the 1-dimensional ranked operator in the direction $\xi$.
\end{property}

Then consider dispersive relations $a(.)=a_0(|.|)$ and similarly for $b$ and $c$, the phase function
$$ \phi(\xi,\eta) = -a(\xi) + b(\eta) + c(\xi-\eta)$$
and for $\sigma \in[0,1]$, $\psi(\xi,\eta,\sigma)=a(\xi)+\sigma \phi(\xi,\eta)$. 

As mentioned in the introduction, the resonant set is then generically of the form
$$ {\mathcal R} = \{(\xi,\eta), |\xi|=R,\ \eta=\lambda \xi\}$$
for some parameters $R>0$ and $\lambda\in\R\setminus\{0,1\}$.

\bigskip We want to discuss here the technical assumption (A3) made in our work :  

\begin{equation} \tag{A3} \left\{\begin{array}{l}
\textrm{for all $\sigma\in[0,1]$, } \operatorname{Hess}_\xi [\psi](\xi,\eta,\sigma) \textrm{ or } \\
 \operatorname{Hess}_{(\xi,\eta)} [\psi](\xi,\eta,\sigma) \textrm{ is non-degenerate on $\mathcal{R}$.}
\end{array}  \right.
\end{equation}

\medskip
First $\operatorname{Hess}_\xi [\psi](\xi,\eta,\sigma) = (1-\sigma)\operatorname{Hess}[a](\xi) + \sigma \operatorname{Hess}[c](\xi-\eta)$ and
\begin{align*}
\lefteqn{\operatorname{Hess}_{(\xi,\eta)}[\psi](\xi,\eta,\sigma) =} & & \\
 & &  \left( \begin{array}{cc}
                                                                                  (1-\sigma)\operatorname{Hess}[a](\xi) +\sigma \operatorname{Hess}[c](\xi-\eta) \ & \ -\sigma \operatorname{Hess}[c](\xi-\eta) \\
         -\sigma \operatorname{Hess}[c](\xi-\eta) \ & \ \sigma (\operatorname{Hess}[b](\eta) + \operatorname{Hess}[c](\xi-\eta))
                                                                                 \end{array} \right).
\end{align*}
Hence, 
\begin{align}
 \lefteqn{\textrm{Det}\operatorname{Hess}_{(\xi,\eta)}[\psi](\xi,\eta)}& & \nonumber\\
 & & = \sigma^d \textrm{Det} \left( \begin{array}{cc}
                                                                                  (1-\sigma)\operatorname{Hess}[a](\xi) \ & \ - \operatorname{Hess}[c](\xi-\eta) \\
         \sigma \operatorname{Hess}[b](\eta) \ & \ (\operatorname{Hess}[b](\eta) + \operatorname{Hess}[c](\xi-\eta))
                                                                                 \end{array} \right). \label{eq:compu} 
\end{align}

\bigskip 

Now let us choose a resonant point $(\xi,\eta)\in{\mathcal R}$, so that $\xi$, $\xi-\eta$ and $\eta$ have the same direction. We have seen that the three Hessian matrices are a linear combination of the identity and of the 1-dimensional range operator in the same direction. By this way all the Hessian matrices commute between themselves and $\operatorname{Hess}[b](\eta) + \operatorname{Hess}[c](\xi-\eta)$ is invertible (by Assumption (A1), since it is equal to $\operatorname{Hess}_\eta[\phi]$). We also deduce that\footnote{This is a standard computation for determinants by blocs. Indeed for $A,B,C,D$ square matrices which commute we have
$$ \textrm{Det}\left(\begin{array}{cc} A & B \\ C & D \end{array}\right) = \textrm{Det}(AD-BC)$$
as soon as $D$ is invertible. This is easily obtained by multiplying the matrix with $\left(\begin{array}{cc} D & 0 \\ -C & \textrm{Id} \end{array}\right)$.}
$$ \textrm{Det}[\operatorname{Hess}_{(\xi,\eta)}[\psi](\xi,\eta)] = \sigma^d \textrm{Det} M(\xi,\eta,\sigma)$$
with
\begin{align*}
 \lefteqn{M(\xi,\eta,\sigma)\overset{def}{=} } & & \\
 & & (1-\sigma)\operatorname{Hess}[a](\xi) \left(\operatorname{Hess}[b](\eta) + \operatorname{Hess}[c](\xi-\eta)\right) + \sigma \operatorname{Hess}[c](\xi-\eta) \operatorname{Hess}[b](\eta). 
\end{align*}

Let us now decompose the space into the component colinear with $\xi$ (which is also colinear with $\eta$ or $\xi-\eta$) and its orthogonal. Since $M(\xi,\eta,\sigma)$ is symmetrical and these two subspaces correspond to eigenvectors : the matrix is invertible if and only if its restrictions onto these two subspaces are invertible.
On $\xi$ (with $\eta=\lambda \xi$ and $|\xi|=R$), we have 
\begin{align*}
 \langle \operatorname{Hess}[a](\xi) \xi,\xi\rangle & = a_0 ''(|\xi|) |\xi|^2 = R^2 a_0''(R) \\
 \langle \operatorname{Hess}[b](\lambda \xi) \xi,\xi\rangle & =  b_0 ''(|\lambda \xi|) |\xi|^2 = R^2 b_0''(|\lambda| R) \\
  \langle \operatorname{Hess}[c]((1-\lambda) \xi) \xi,\xi\rangle & = c_0 ''(|(1-\lambda) \xi|) | \xi|^2 = R^2 c_0''(|1-\lambda| R).
\end{align*}
So the symmetrical matrix $M(\xi,\eta,\sigma)$ is invertible in $\xi$ if and only if
\begin{equation}
\label{invert1}
(1-\sigma) a_0''(R) \left[ b_0''(|\lambda| R) + c_0''(|1-\lambda| R) \right] + \sigma  b_0''(|\lambda| R) c_0''(|1-\lambda| R) \neq 0. 
\end{equation}
On the orthogonal space $E\overset{def}{=}\xi^\perp=\eta^\perp$, we have
\begin{align*}
  \operatorname{Hess}[a](\xi)_{|E}  & =\frac{a_0 '(|\xi|)}{|\xi|} \textrm{Id}_{|E} = \frac{a_0 '(R)}{R} \textrm{Id}_{|E} \\
 \operatorname{Hess}[b](\lambda \xi) &= \frac{b_0 '(|\lambda \xi|)}{|\lambda \xi|} \textrm{Id}_{|E} = \frac{b_0 '(|\lambda| R)}{|\lambda| R} \textrm{Id}_{|E} \\
 \operatorname{Hess}[c]((1-\lambda) \xi) & = \frac{c_0 '(|(1-\lambda)\xi|)}{|(1-\lambda)\xi|} \textrm{Id}_{|E} = \frac{c_0 '(|1-\lambda|R)}{|1-\lambda| R} \textrm{Id}_{|E}.
\end{align*}
So the matrix $M(\xi,\eta,\sigma)$ is invertible on $E$ if and only if
$$ (1-\sigma) a_0 '(R) \left[ \frac{b_0 '(|\lambda| R)}{|\lambda|} + \frac{c_0 '(|1-\lambda|R)}{|1-\lambda|} \right] + \sigma \frac{b_0 '(|\lambda| R)}{|\lambda|} \frac{c_0 '(|1-\lambda|R)}{|1-\lambda|} \neq 0 $$
which is equivalent to 
\begin{equation}
\label{invert2}
(1-\sigma) a_0 '(R) \left[ |1-\lambda| b_0 '(|\lambda| R) + |\lambda|c_0 '(|1-\lambda|R) \right] + \sigma b_0 '(|\lambda| R) c_0 '(|1-\lambda|R) \neq 0.
\end{equation}
We also deduce the following property in the current context: $M(\xi,\eta,\sigma)$ is non-degenerate on the resonant set if and only if (\ref{invert1}) and (\ref{invert2}) hold.

Finally, it is easy to see that $\operatorname{Hess}_\xi [\psi](\xi,\eta,\sigma)$ is non-degenerate on $\mathcal{R}$ if and only if
\begin{equation} (1-\sigma) a_0''(R) + \sigma c_0''(|1-\lambda| R) \neq 0 \label{invert3} \end{equation}
and
\begin{equation} (1-\sigma)|1-\lambda| a_0 '(R) + \sigma c_0 '(|1-\lambda|R) \neq 0. \label{invert4} \end{equation}

\bigskip We note that if (\ref{invert3}) is not satisfied then (\ref{invert1}) holds and if (\ref{invert4}) is not satisfied then (\ref{invert2}) holds.
To summarize all these observations: 

\begin{property} \label{prop} For a resonant set given by $R>0$ and $\lambda \neq 1$, the assumption (A3) is equivalent to the following: for all $\sigma \in[0,1]$, (\ref{invert1}) and (\ref{invert2}) hold or (\ref{invert3}) and (\ref{invert4}) hold.
\end{property}

It follows from this property that the hypothesis (A3) is generically satisfied.

\bigskip

\noindent {\bf Acknowledgements:} The second author is grateful to Beno\^it Pausader for enlightening conversations during the writing of this article.

\end{document}